\documentclass[journal,web]{ieeecolor}
\usepackage{generic}
\usepackage{cite}
\usepackage{amsmath}
\usepackage{amssymb,bm}
\usepackage{theoremref}
\usepackage{color}
\usepackage{apptools}
\usepackage{graphicx}

\newtheorem{thm}{Theorem}

\newtheorem{prop}{Proposition}
\newtheorem{cor}{Corollary}
\newtheorem{defn}{Definition}
\newtheorem{asmp}{Assumption}

\newtheorem{exmp}{Example}
\newtheorem{rem}{Remark}

\DeclareBoldMathCommand{\grad}{\mathrm{grad}}
\DeclareBoldMathCommand{\hess}{\mathrm{Hess}}
\newcommand{\Real}{\mathbb{R}}
\newcommand{\Man}{\mathcal{M}}
\newcommand{\Dom}{\mathcal{U}}
\newcommand{\Ei}{\mathcal{E}}

\begin{document}
	\title{Non-Holonomic Gradient Play: Leafwise Nash Equilibria, Stability, and Deception}
	\author{Mahmoud Abdelgalil, Miroslav Krstic, and Jorge I. Poveda
		\thanks{M. Abdelgalil ({\tt\small maabdelg@buffalo.edu}) is with the Department of Mechanical and Aerospace Engineering, University at Buffalo, State University of New York, Buffalo, NY, 14260, USA. 
			Miroslav Krstic ({\tt mkrstic@ucsd.edu}) is with the Department of Mechanical and Aerospace Engineering, University of California San Diego, La Jolla, CA 92093, USA. 
			J. I. Poveda ({\tt\small jipoveda@ucsd.edu}) is with the Department of Electrical and Computer Engineering,
			University of California San Diego, La Jolla, CA 92093, USA. This work was supported in part by DARPA under Grant No. HR00112530225.
		}
	}
	
	\maketitle
	\pagestyle{empty}
	\thispagestyle{empty}
	\begin{abstract}
		We study generalized learning dynamics in multi-agent systems whose joint state evolves on a manifold and whose agents act through state-dependent, potentially nonholonomic vector fields. Under a bundle-splitting condition, we show that these dynamics admit an intrinsic representation as projected
		Riemannian gradients, giving rise to a class of \emph{nonholonomic gradient play} dynamics. We characterize the local stability of its equilibria through an intrinsic linearization that explicitly captures the effects of the Riemannian connection and the nonholonomy of the actuation frame. The framework recovers classical gradient play on Euclidean spaces and manifolds as special cases while accommodating nonholonomic actuation. We then develop a geometric theory of deception under asymmetric information, whereby an
		agent exploits privileged knowledge of other agents' learning rules to
		manipulate the emerging equilibrium. We show that deception effectively
		\emph{tilts the Riemannian geometry} perceived by the oblivious agents,
		distorting their projected gradient directions. We establish persistence
		of exponentially stable equilibria over an open set of deception parameters and derive an explicit first-order characterization of the resulting equilibrium displacement. Analytical and numerical examples illustrate the framework and its implications for equilibrium manipulation in games.
	\end{abstract}
	

	\section{INTRODUCTION}
	\IEEEPARstart{S}{trategic} interactions between multi-agent autonomous decision-making systems are often constrained by geometry, with examples spanning many practical engineering problems, e.g., rigid body attitude coordination \cite{sarlette2009autonomous,ren2009distributed}, distributed resource allocation \cite{poveda2013distributed,ochoa2019hybrid}, and synchronization of coupled oscillators \cite{yin2011synchronization}. These settings have two distinguishing structural features. First, the joint state of the interaction evolves on a manifold rather than on a vector space. Second, each agent can influence a restricted set of directions of the joint state, and those admissible directions typically vary with the state itself, i.e., each agent acts through a family of state-dependent control vector fields. If, in addition, the agents pursue conflictive objectives described by individual payoff functions, the natural model of their interaction is given by \emph{gradient play dynamics} wherein each agent seeks to maximize its own payoff along the vector fields it controls \cite{shamma2005dynamic,toonsi2024higher,li2025passivity}, typically leading to convergence to Nash equilibria. 
	
	The earliest stability analysis of gradient play dynamics in games with continuous strategy spaces seems to be \cite{ratliff2016characterization,mazumdar2020gradient}, where differential characterizations of equilibria and stability are given for games whose strategy spaces are finite-dimensional manifolds. However, the learning dynamics considered therein are expressed via partial derivatives of the payoffs with respect to a given choice of coordinates. Since gradient vector fields on a manifold are defined only relative to a Riemannian metric, and no such metric is introduced in \cite{ratliff2016characterization,mazumdar2020gradient}, neither the dynamics nor the ensuing stability conditions are coordinate independent. Moreover, existing works in the literature typically assume that each agent actuates the coordinates of its own strategy space directly and independently, a framework that does not accommodate agents that act through a restricted family of vector fields, let alone a non-commuting one, even if the family collectively spans the tangent bundle of the manifold. Such structures arise naturally in nonholonomic and underactuated systems, including wheeled mobile robots, autonomous vehicles, robotic manipulators, and aerial and marine vehicles, where admissible motions are constrained by the system's actuation geometry. While the literature on learning in games provides a rich set of tools for
	dynamics in flat spaces, typically $\mathbb{R}^n$ or constrained subsets
	thereof in generalized Nash equilibrium problems \cite{belgioioso2022distributed}, to the best of our
	knowledge, no intrinsic theory exists for games whose joint state evolves
	on a manifold and whose agents act through state-dependent, possibly
	noncommuting vector fields.
	
	
	Motivated by the previous background, the first contribution of the present manuscript is precisely an \emph{intrinsic theory of gradient play dynamics} for multi-agent systems evolving on smooth manifolds. Specifically, under a natural bundle-splitting assumption on the agents' actuation directions, we construct a Riemannian metric on the state space with respect to which the interaction dynamics amounts to a sum of \emph{projected Riemannian gradients}, one per agent, with each projection being onto the distribution spanned by the corresponding agent's control vector fields. We refer to the resulting dynamics as \emph{non-holonomic gradient play}, and to the corresponding vector field as the \emph{Riemannian pseudogradient flow}. This representation yields an equilibrium characterization that is coordinate independent. In particular, equilibria are determined solely by the payoffs of the agents and the geometry of the actuation directions. We then characterize the structural and the asymptotic stability properties of equilibria via an intrinsic linearization of the non-holonomic pseudo-gradient, which decomposes into two types of terms: Riemannian Hessian terms, and \emph{Connection} terms. The latter type encodes the non-holonomy of the actuation frame and has no counterpart in the coordinate-based analysis considered in the game-theoretic literature \cite{ratliff2016characterization,mazumdar2020gradient}. Since typical coordinate-dependent notions of Nash equilibria \cite{nash1951non,kristaly2014nash} are inapplicable in our setting, we introduce the notion \emph{Leafwise Nash equilibria} to capture the game-theoretic properties of the underlying equilibrium points of the dynamics. Sufficient stability conditions of the block Gershgorin type are then derived, and two special cases are worked out in detail. First, when each agent's distribution is involutive, we show that the diagonal blocks of the linearization reduce to intrinsic Hessians of restricted payoffs. When, in addition, the agents' vector fields mutually commute, the state space acquires a local product structure and the characterizations of \cite{ratliff2016characterization,mazumdar2020gradient} are recovered as the flat special case. The identical interest game \cite{monderer1996fictitious}, wherein the dynamics reduce to an intrinsic gradient flow, is likewise subsumed.
	
	The second main contribution of this paper concerns \emph{deception} under asymmetric information, and it is here that the generality of the first contribution proves instrumental. Specifically, in contested or adversarial environments, agents need not be
	equally informed: privileged agents may have knowledge of the learning rules implemented by their competitors. Recent works
	\cite{tang2025deception,tang2025deception1,tang2025stochastic,
		tutorialCDC2026,vamvoudakis2026tutorial} have shown that, in Euclidean product strategy spaces with additive actuation, such information asymmetries can be exploited to manipulate the emerging equilibria without tampering with other agents' measurements or control actions. These mechanisms, however, rely on
	the additive structure of $\mathbb{R}^n$ and do not directly extend to the geometric setting considered here.
	
	In contrast, within our framework, we show that deception admits an intrinsic geometric interpretation: \emph{by strategically modifying its own control signals, a deceptive agent can ``tilt'' the oblivious agent's effective actuation bundle along its own, thereby distorting the Riemannian projections that govern its nonholonomic gradient dynamics}.  These tilted bundles may fail to be involutive even when their nominal counterparts are, making the general nonholonomic theory developed in the first part of the paper essential for the analysis of deception. Using this framework, we show that exponentially stable equilibria, together with their stability, persist over an open set of deception parameters, and derive an explicit first-order formula for the resulting equilibrium displacement. To the best of our knowledge, this work provides the first geometric-control analysis of nonholonomic gradient play in noncooperative games defined on manifolds and its associated geometric notions of Nash equilibrium, and, independently, the first intrinsic geometric characterization of deception in multi-agent learning dynamics.
	
	The remainder of the paper is organized as follows. Section~II introduces
	the notation. Section~\ref{sec:prob_formulation}
	presents the problem formulation.
	Section~\ref{sec:no_deception_stability} presents the main results under symmetric information.
	Section~\ref{sec:deception_stability} studies deception. Section~VI contains the proofs, and
	Section~VII concludes the paper.
	\section{PRELIMINARIES AND NOTATION}
	Throughout, we use standard differential geometric notation consistent with, e.g., \cite{lee2012smooth}. Let $\mathcal{M}$ be a smooth $n$-dimensional manifold, $T\mathcal{M}$ be its tangent bundle, $T^*\mathcal{M}$ its cotangent bundle. We use $\Gamma(T\mathcal{M} \otimes \cdots \otimes T\mathcal{M} \otimes T^*\mathcal{M} \otimes \cdots\otimes T^*\mathcal{M})$ to denote the space of smooth sections of the bundle $T\mathcal{M} \otimes \cdots \otimes T\mathcal{M} \otimes T^*\mathcal{M} \otimes \cdots \otimes T^*\mathcal{M}$, so that the space of smooth vector fields is $\Gamma(T\mathcal{M})$, the space of smooth differential $1$-forms is $\Gamma(T^*\mathcal{M})$, and so on.
	For brevity, we also use $\mathfrak{X}(\mathcal{M})$ to denote the space of smooth vector fields on $\mathcal{M}$, and $\Omega(\mathcal{M})$ to denote the space of smooth differential $1$-forms on $\mathcal{M}$.
	The symbol $\otimes$ denotes \emph{tensor product} so that, if $f_1,f_2\in\mathfrak{X}(\mathcal{M})$ and $\omega_1,\omega_2\in\Omega(\mathcal{M})$, then $f_1\otimes f_2\in \Gamma(T\mathcal{M}\otimes T\mathcal{M})$, $\omega_1\otimes\omega_2\in \Gamma(T^*\mathcal{M}\otimes T^*\mathcal{M})$, $\omega_i\otimes f_i \in \Gamma(T^*\mathcal{M}\otimes T\mathcal{M})$, and so on. For $f\in \mathfrak{X}(\mathcal{M})$ and $J$ a smooth function, the Lie derivative of $J$ along $f$ is $L_f J(x) := dJ(x)\big[f(x)\big]$, where $dJ\in\Omega(\mathcal{M})$ denotes the differential of $J$. A Riemannian metric on $\mathcal{M}$, which is a symmetric positive definite section of $\Gamma(T^*\mathcal{M}\otimes T^*\mathcal{M})$, is denoted by $g$.
	Given a smooth function $J$, the Riemannian gradient with respect to $g$, denoted by $\grad(J)$, is the unique vector field satisfying $g(\grad(J), X) = dJ[X], \quad \forall X\in \mathfrak{X}(\mathcal{M})$. The Riemannian \emph{Hessian} of $J$ is the unique symmetric section of $\Gamma(T^*\mathcal{M}\otimes T^*\mathcal{M})$ that satisfies $\hess(J)[X,Y]:=g(\nabla_X\grad(J), Y)$, where $\nabla:\mathfrak{X}(\mathcal{M})\times\mathfrak{X}(\mathcal{M})\rightarrow \mathfrak{X}(\mathcal{M})$ is the Levi-Civita connection associated with the metric $g$. The fact that the Riemannian Hessian is symmetric, i.e., $\hess(J)[X,Y]=\hess(J)[Y,X]$, follows from the torsion-free nature of the Levi-Civita connection. If $\mathcal{E}\subset T\mathcal{M}$ is a smooth sub-bundle and $P_{\mathcal{E}}:T\mathcal{M}\to \mathcal{E}$ is the $g$-orthogonal projection, then we define the \emph{projected gradient} $\grad_{\mathcal{E}}(J) := P_{\mathcal{E}}\,\grad(J)$. Finally, we use $\cdot|_x$ to denote the evaluation operation, i.e., if $f\in \mathfrak{X}(\mathcal{M})$, then $f|_{x}=f(x)\in T_x\mathcal{M}$, and so on.
	\section{PROBLEM FORMULATION}\label{sec:prob_formulation}
	In this section, we describe the problem formulation, as well as its connections to the literature on learning in games. 

	\vspace{-0.3cm}
	\subsection{Geometric Learning in Games}
	We consider $N$ independent decision-making agents interacting through a strategic environment modeled as a smooth $n$-dimensional manifold $\mathcal{M}$. The $i^{\rm th}$ agent influences the environmental state $x\in\mathcal{M}$ through a collection of vector fields $\{f_{i,j}\}_{j=1}^{d_i}\subset\mathfrak{X}(\mathcal{M})$, with its actions determined by the control inputs $(u_{i,1},\ldots,u_{i,d_i})\in\mathbb{R}^{d_i}$. These inputs act on the environment through a control-affine structure, yielding the dynamics
	\begin{align}\label{eq:env_model}
		\dot{x}&=\sum_{i=1}^N\sum_{j=1}^{d_i} u_{i,j} f_{i,j}(x),~~~~x\in\mathcal{M}.
	\end{align}
	In the non-cooperative setting, each agent is interested in optimizing their own payoff, which may, in general, be different from other agents' payoffs. We model this setting by assigning to each agent a sufficiently regular function $J_i$, for $i\in\{1,2,\ldots,N\}$. In this noncooperative setting, a typical objective of the agents is to converge to a Nash equilibrium \cite{rosen1965existence}. A block-diagram representation of the class of systems considered in this paper is shown in Figure~\ref{fig:env_model}.
	\begin{figure}
		\centering
		\includegraphics[width=0.8\linewidth]{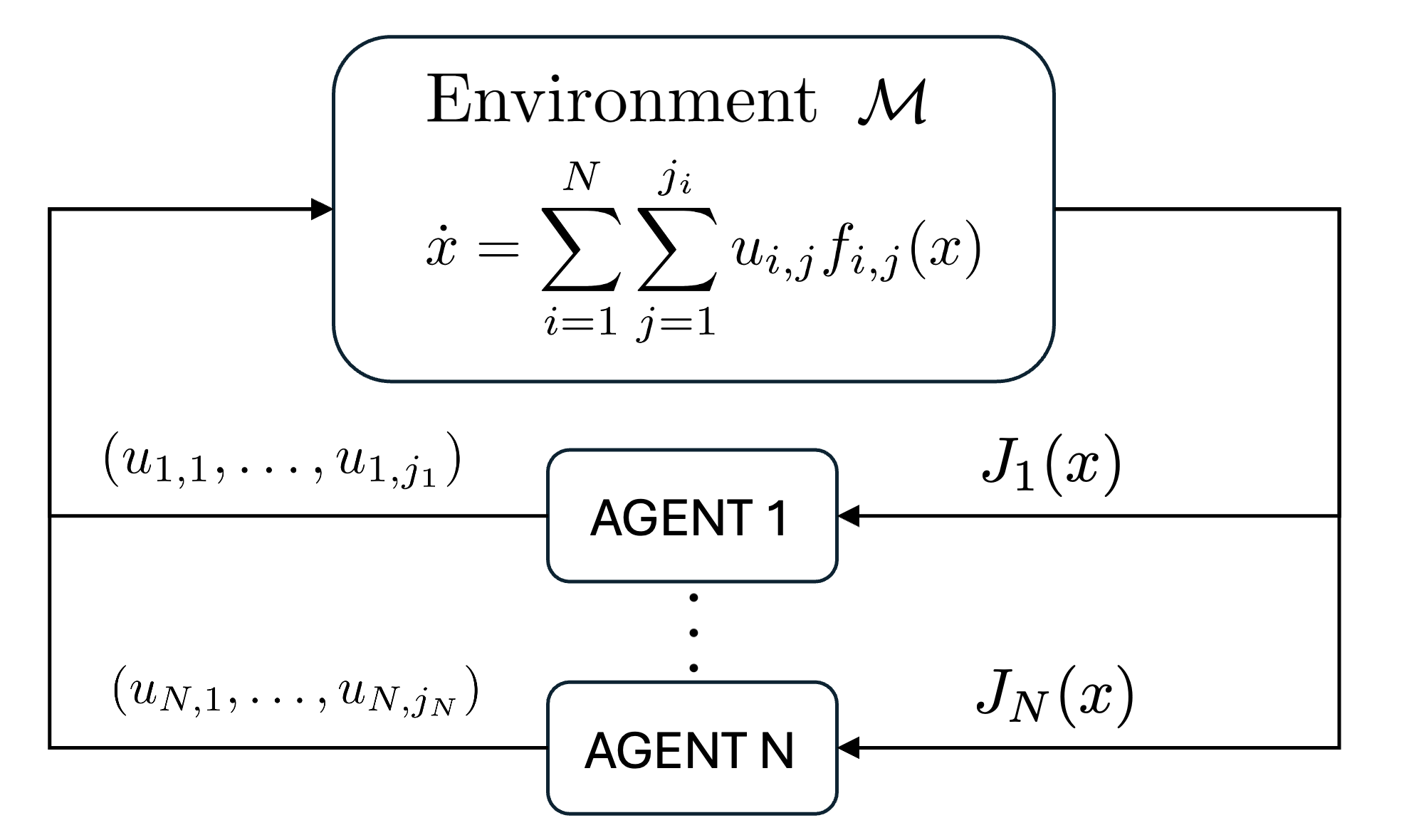}
		\caption{Block diagram of the strategic interaction between $N$ independent decision-making agents coupled only through the environment.}
		\label{fig:env_model}
	\end{figure}

	In multi-agent systems evolving in $\mathbb{R}^n$, to converge to a Nash equilibrium, it is common to consider decentralized feedback laws $u_i$ based on the partial gradient (i.e., pseudogradient) information of the cost functions of the agents, i.e., gradient play \cite{shamma2005dynamic,toonsi2024higher}. For the control-affine structures of the form \eqref{eq:env_model}, the natural extension is given by the feedback law:
	\begin{equation}
		u_{ij}(x)=\kappa_i L_{f_{i,j}}(J_i)(x),
	\end{equation}
	where $\kappa_i>0$ is a tunable gain. This choice leads to the following dynamics:
	\begin{align}\label{eq:liebra_system}
		\dot{x}&= \sum_{i=1}^N\kappa_{i} \sum_{j=1}^{d_i} L_{f_{i,j}}(J_i)(x) f_{i,j}(x), ~~~x\in\mathcal{M}.
	\end{align}
	System \eqref{eq:liebra_system} is the main object of study of this paper. As shown in the next section, under mild regularity conditions
	on $\mathcal{M}$, $\{f_{i,j}\}$, and $J_i$, these dynamics admit an intrinsic
	interpretation as nonholonomic gradient play, whose equilibria define a
	generalized notion of Nash equilibrium that accounts for the geometric
	constraints on the players' feasible actions.
	\begin{rem}
		Dynamics of the form \eqref{eq:liebra_system} arise naturally in multi-agent
		systems where agents optimize individual objectives subject to physical or
		kinematic actuation constraints. Examples include nonholonomic mobile robots,
		autonomous vehicles, robotic manipulators, and spacecraft, whose configurations
		evolve on nonlinear manifolds and whose feasible motions are restricted to
		agent-specific distributions $\mathcal{E}_i$. In such settings, an agent cannot
		generally move along the full gradient of its payoff and must instead exploit
		the components of the gradient along its admissible directions. 
		\hfill $\square$
	\end{rem}
	\begin{rem}[Nonholonomic vs Standard Gradient Play]
		When $\mathcal{M}=\mathbb{R}^n$, and the agents' vector fields are constant and decoupled, system \eqref{eq:liebra_system} recovers the standard \emph{gradient-play dynamics} $\dot{x}_i=\nabla_i J_i$ \cite{li2025passivity,shamma2005dynamic}. While the convergence and stability properties of standard gradient play are well understood, the more general manifold-constrained dynamics \eqref{eq:liebra_system}, particularly in the presence of state-dependent and potentially non-commuting actuation directions, have remained unexplored in the context of games. \hfill $\square$
	\end{rem}
	
	\begin{rem}[Model-Free Implementation]\label{remark1}
		Systems of the form \eqref{eq:liebra_system} also emerge in the analysis of \emph{model-free} (i.e., payoff-based) decision-making algorithms, either via Lie-bracket deterministic approximations ~\cite{durr2013liebracket,scheinker2013model,tutorialCDC2026} or Lie-bracket stochastic approximations \cite{LieBraSA}. In the former case, \eqref{eq:liebra_system} emerges as the average system obtained via the oscillatory control law 
		\begin{equation}\label{eq:es_law}
			\begin{aligned}
				u_{i,j}&=c_{i,j}\left(  \cos(\iota_{ij} t) F^1_{i,j}(J_i(x))- \sin(\iota_{ij} t) F^2_{i,j}(J_i(x))\right),
			\end{aligned}
		\end{equation}
		where $c_{ij}=\sqrt{2\kappa_{i} \iota_{ij}}$, $\iota_{ij}=2\pi\omega_{ij}\omega$, $\omega\in\mathbb{R}_{>0}$ is a tunable parameter, the constants $\cup_{i=1}^N\{\omega_{i,j}\}_{j=1}^{d_i}\subset \mathbb{Q}_{>0}$ are any collection of positive rational numbers such that $\omega_{i,j}=\omega_{i',j'}$ if and only if $i=i'$ and $j=j'$, $\kappa_{i}>0$, and $F^k_{i,j}:\mathbb{R}\rightarrow\mathbb{R}$ is any collection of sufficiently regular functions that satisfy
		\begin{align}\label{eq:es_integral_idnty}
			F^2_{i,j}(y)&= - F^1_{i,j}(y)\int F^1_{i,j}(y)^{-2}\mathrm{d}y.
		\end{align}
		An example of such functions is $F^1_{i,j}(y):=\cos(y)$ and $F^2_{i,j}(y):=-\sin(y)$  \cite{scheinker2013model}; see also \cite{scheinker2014extremum,grushkovskaya2018class}.  In this case, as $\omega\to\infty$, the behavior of the closed-loop system is predicted precisely by system \eqref{eq:liebra_system}. In the context of stochastic approximation \cite{LieBraSA}, system \eqref{eq:liebra_system} arises as the limiting dynamics as the interpolation step size vanishes \cite{LieBraSA}.
		\hfill $\square$
	\end{rem}
	\begin{figure}[t]
		\centering
		\includegraphics[width=0.9\linewidth]{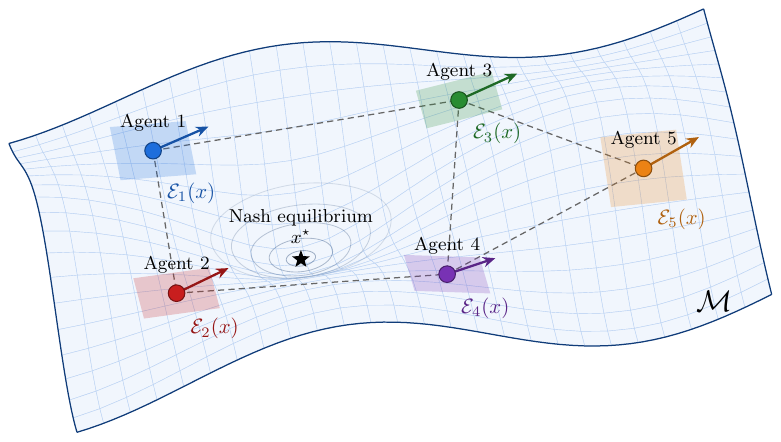}
		\caption{Illustration of the environment state space $\mathcal{M}$ and the multi-agent system interactions via the individual distributions $\mathcal{E}_i(x)=\mathrm{span}_{\mathbb{R}}\{f_{i,j}(x)\}_{j=1}^{d_i}$ (c.f. Assumption 1).}
		\vspace{-0.4cm}
		\label{figure:manifold}
	\end{figure}
	%
	%
	%
	
	\vspace{-0.2cm}
	\subsection{Asymmetric Information and Deception}\label{sec:asym}
	The model developed thus far assumes \emph{information symmetry}, whereby no
	agent has \emph{privileged} knowledge of the feedback laws implemented by
	others. Recent works \cite{tang2025deception,tang2025deception1,tang2025stochastic}
	show that violating this assumption can fundamentally alter the asymptotic
	behavior of multi-agent systems. In particular, under \emph{information
		asymmetry}, an informed agent can strategically modify its own control inputs
	to bias the steady-state behavior of \eqref{eq:env_model}, without directly
	altering the actions or measurements of the \emph{oblivious agents}. We refer
	to such an informed agent as \emph{deceptive}, see the recent tutorials 
	\cite{vamvoudakis2025deception,vamvoudakis2026tutorial}. However, the context in which the phenomenon of  \emph{deception} has been studied in the literature of Nash-Seeking is that of repeated static non-cooperative games on \emph{flat spaces}, particularly, in $\mathbb{R}^n$. Nevertheless, deception in generalized equilibrium-seeking problems on smooth manifolds remains largely unexplored, particularly \emph{with respect to its interplay with the underlying geometry}. Based on this, the second main objective of this paper is to extend the study of deception in adaptive multi-agent decision-making to the geometric setting described by \eqref{eq:env_model}.
	\begin{rem}[Deceptive Nonholonomic Gradient Play]
		To illustrate how asymmetric information can be leveraged to bias the
		asymptotic behavior of the target system \eqref{eq:liebra_system}, we
		consider the setting of Lie-bracket approximations  \footnote{The particular approximation mechanism is immaterial to results of this paper, as the target system \eqref{eq:liebra_system} can arise from different classes of learning algorithms, including those based on Lie-bracket approximations \cite{durr2013liebracket,grushkovskaya2018class}, stochastic approximations \cite{LieBraSA}, etc.} discussed in Remark \ref{remark1}. For simplicity, we consider a two-agent setting (the extension to $N>2$
		agents is discussed in Remark~\ref{rem:beyond_two_agents}), and we refer to
		the $1^{\rm st}$ agent, as the \emph{deceptive agent}. We assume that this agent has
		access to the control input \eqref{eq:es_law} generated by the
		$2^{\rm nd}$ agent, referred to as the \emph{oblivious agent}. By
		leveraging this privileged information, the deceptive agent modifies its
		feedback law according to
		\begin{align}\label{eq:deceptive_es_law}
			\tilde{u}_{1,j}(x,t)
			:=
			u_{1,j}(x,t)
			+
			\sum_{k=1}^{d_2}\gamma_{k}^{j}u_{2,k}(x,t),
		\end{align}
		where $\{\gamma_{k}^{j}\}$ is a collection of tunable deception parameters
		selected by the deceptive agent. Under this modification, the resulting
		closed-loop system can be equivalently written as
		\begin{subequations}
			\begin{align}\label{eq:closed_loop_deception}
				\dot{x}
				&=
				\sum_{i=1}^{2}\sum_{j=1}^{d_i}
				u_{i,j}(x,t)\tilde{f}_{i,j}(x),
			\end{align}
			where
			\begin{align}\label{eq:tilted_directions}
				\tilde{f}_{1,j}
				&:=f_{1,j},
				\qquad
				\tilde{f}_{2,j}
				:=
				f_{2,j}
				+
				\sum_{k=1}^{d_1}\gamma_{j}^{k}f_{1,k}.
			\end{align}
		\end{subequations}
		Consequently, using standard averaging tools \cite{durr2013liebracket,scheinker2013model,grushkovskaya2018class}, it can be shown that the associated Lie-Bracket approximation system is now given by
		\begin{equation}\label{eq:liebra_sys_deception}
			\dot{x}
			=
			\sum_{i=1}^{2}\kappa_i
			\sum_{j=1}^{d_i}
			L_{\tilde{f}_{i,j}}(J_i)(x)\tilde{f}_{i,j}(x).
		\end{equation}
		Thus, deception preserves the structure of \eqref{eq:liebra_system} while \emph{tilting} the oblivious agent's vector fields according to \eqref{eq:tilted_directions}. In this way, by exploiting knowledge of the oblivious agent's control law, the deceptive agent reshapes the effective directions of its payoff dynamics, generally shifting the equilibria of \eqref{eq:liebra_sys_deception}. The nominal dynamics are recovered when $\gamma_j^k\equiv 0$. In Section~\ref{sec:deception_stability}, we will rigorously characterize this scenario. \hfill $\square$
	\end{rem}
	\section{NONHOLONOMIC GRADIENT PLAY}\label{sec:no_deception_stability}
	%
	%
	%
	%
	In this section, we study the equilibria and asymptotic behavior of system \eqref{eq:liebra_system}.

	\vspace{-0.3cm}
	\subsection{Riemannian Pseudogradient Flow and its Equilibria}
	Without further assumptions on the vector fields $f_{i,j}$, finding the equilibria of \eqref{eq:liebra_system} requires solving the general system of equations
	\begin{align}\label{eq:eqbm_coupled}
		\sum_{i=1}^N \kappa_{i} \sum_{j=1}^{d_i} L_{f_{i,j}}(J_i)(x) f_{i,j}(x)&=0.
	\end{align}
	A limitation of \eqref{eq:eqbm_coupled} is that its equilibria generally
	depend on the gains $\kappa_i$, which are chosen independently of the payoff
	functions $J_i$. Nevertheless, there are natural assumptions that allow us to characterize the equilibria of \eqref{eq:eqbm_coupled} for any choice of $\kappa_{i}$. To that end, we impose the following assumption.
	\begin{asmp}\label{asmp:indpenendent_families}
		There exists an open sub-manifold $\mathcal{U}\subset\mathcal{M}$ such that, for all $x\in\mathcal{U}$, we have
		\begin{align}\label{eq:tangent_splitting}
			T_x\mathcal{U} = \mathcal{E}_1(x)\oplus \cdots\oplus \mathcal{E}_N(x),
		\end{align}
		where $\mathcal{E}_i$ is a \emph{regular} distribution on $\mathcal{U}$ defined by the span of the $i^{\rm th}$-family of vector fields $\{f_{i,j}\}_{j=1}^{d_i}$, i.e., $\mathcal{E}_i(x):=\mathrm{span}_{\mathbb{R}}\{f_{i,j}(x)\}_{j=1}^{d_i}$. \hfill $\Box$
	\end{asmp}
	\medskip
	In words, Assumption \ref{asmp:indpenendent_families} asks that the family of vector fields $\cup_{i=1}^{N}\{f_{i,j}\}_{j=1}^{d_i}$ spans the tangent space of the manifold $\mathcal{M}$ for all points in an open set $\mathcal{U}$ and that the intersection of the spans of any two distinct families $\{f_{i,j}\}_{j=1}^{d_i}$ and $\{f_{i',j}\}_{j=1}^{d_{i'}}$ is trivial whenever $i\neq i'$.
	In the context of the multi-agent strategic interaction modeled by \eqref{eq:env_model}, Assumption \ref{asmp:indpenendent_families} essentially requires that the control vector fields influenced by the $i^{\rm th}$-agent are linearly independent from the span of the control vector fields influenced by every other agent at every environment state $x\in\mathcal{U}$ and that, collectively, the family $\cup_{i=1}^{N}\{f_{i,j}\}_{j=1}^{d_i}$ spans the entire tangent bundle $T\mathcal{U}$. We now illustrate this assumption with some examples. In these examples, the distributions $\mathcal{E}_i$ are all rank $1$, i.e., $d_i=1$ for all $i\in\{1,\dots,N\}$. To simplify notation, we use $e_i:=f_{i,1}$ to get rid of superfluous indices.

	\vspace{0.1cm}
	\begin{exmp}\label{exmp:sphere_exmp_no_deception}
		\begin{figure}[t]
			\centering
			\includegraphics[width=0.8\linewidth]{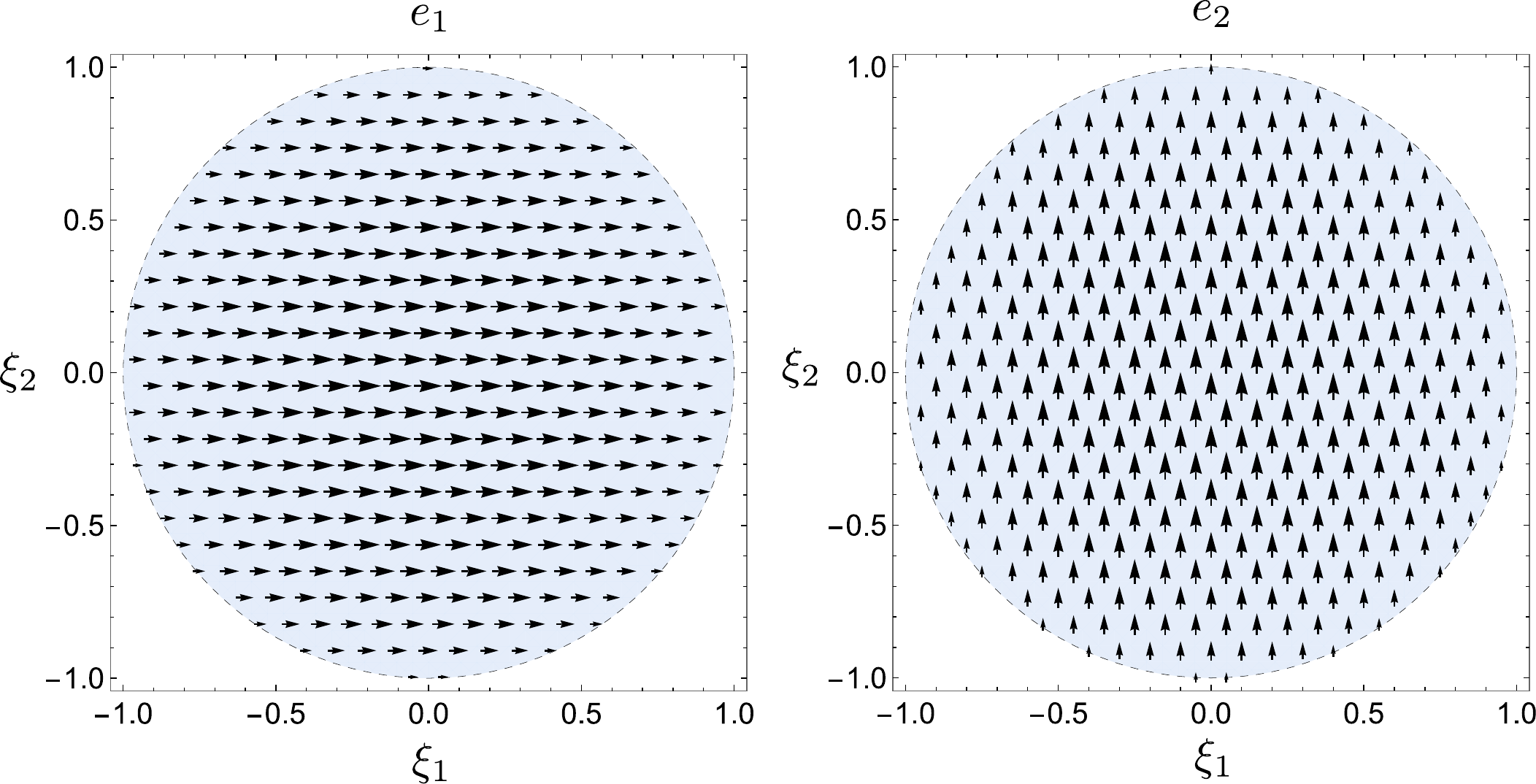}
			\caption{\normalfont An illustration of the vector fields $e_{i}$ for Example \ref{exmp:sphere_exmp_no_deception} in the $\xi$-coordinates. The arrows' sizes represent the magnitude of the vector field.}
			\label{fig:sphere_example_vectorfields}
		\end{figure}
		We consider a two-player decision-making problem on the sphere $\mathcal{M}:=\mathbb{S}^2=\{x\in\mathbb{R}^3~|~x^\top x = 1\}\subset\mathbb{R}^3$, which is a $2$-dimensional Euclidean sub-manifold.
		Let the vector fields $e_{i}$ be defined in the ambient coordinates $x=(x_1,x_2,x_3)$ by $e_i(x):=x_3 \partial_{i} - x_i\partial_3$,
		which can be shown to be tangent to $\mathbb{S}^2$.
		To verify Assumption \ref{asmp:indpenendent_families} on a suitable submanifold $\mathcal{U}\subset\mathcal{M}$, we define the function
		$
		\Delta(x):=\langle x, e_{1}(x)\times e_{2}(x) \rangle_{\mathbb{R}^3},
		$
		where $\times$ is the cross product of vectors and $\langle\cdot,\cdot\rangle_{\mathbb{R}^3}$ is the inner product between vectors in $\mathbb{R}^3$. Direct computation gives that $\Delta(x) = x_3(x_1^2+x_2^2+x_3^2) = x_3$.
		Therefore, since the cross product of a pair of vectors in $\mathbb{R}^3$ vanishes if and only if they are linearly dependent, it follows that the vector fields $e_{i}$ are linearly independent and span the tangent space of $\mathbb{S}^2$ everywhere except at the degenerate set $\{x\in\mathbb{S}^2~|~x_3=0\}$. In particular, by restricting our analysis to the open sub-manifold
		$
		\mathcal{U} := \{x\in\mathbb{S}^2~|~x_3>0\},
		$
		we are guaranteed that, if $\mathcal E_i(x):=\mathrm{span}\{e_i(x)\}$, then Assumption \ref{asmp:indpenendent_families} holds on $\mathcal{U}$.
		Moreover, the submanifold $\mathcal{U}$ is diffeomorphic to the open disk $\mathbb{D}^2:=\{\xi=(\xi_1,\xi_2)\in\mathbb{R}^2~|~\xi_1^2+\xi_2^2<1\}$ via the diffeomorphism $(\xi_1,\xi_2)= \varphi(x)= (x_1,x_2)$, which will be convenient for visualizations.
		For instance, the vector fields $e_{i}$ are depicted in Figure \ref{fig:sphere_example_vectorfields} using the $\xi$-coordinates. It is also worth mentioning that the vector fields $e_{i}$ do not commute. Indeed, direct computation gives $[e_{1},e_{2}](x) = x_2\partial_1 - x_1\partial_2$, which, generically, is non-vanishing on $\mathcal{U}$.
		\phantom{.}\hfill $\Box$
	\end{exmp}
	\vspace{0.1cm}
	\begin{exmp}\label{exmp:simplex_exmp_no_deception}
		Consider a two-player  resource allocation problem on the $2$-simplex $\mathcal{M}=\mathcal{S}^2:=\{x\in\mathbb{R}^3~|~x_i\geq 0, \sum_{i=1}^3x_i=1\}$, which is a 2-dimensional Euclidean submanifold. Consider the vector fields $e_i$ defined in the ambient coordinates as follows: $e_i(x):=x_{1}x_{2}x_{3}(\partial_{i}-\partial_3)$. Direct computation shows that the vector fields $e_i$ are tangent to $\mathcal{M}$ and span the tangent space everywhere except at the degenerate set $\{x\in\mathcal{M}~|~ x_1x_2x_3=0\}$, which corresponds to the topological boundary of $\mathcal M$, denoted by $\partial \mathcal{M}$. Hence, by restricting our analysis to the open submanifold
		$
		\mathcal{U}:= \mathcal{M}\backslash\partial\mathcal{M},
		$
		we are guaranteed that, with $\mathcal{E}_i=\mathrm{span}\{e_i\}$, then Assumption \ref{asmp:indpenendent_families} holds on $\mathcal{U}$. Moreover, similar to the previous example, the submanifold $\mathcal{U}$ is diffeomorphic to an open subset of the plane via the diffeomorphism $(\xi_1,\xi_2)= \varphi(x)= (x_1,x_2)$, which will be convenient for visualizations. Indeed, the vector fields $e_i$ are depicted in Figure \ref{fig:P2_exmp_vectorfields}.
		\begin{figure}[t]
			\centering
			\includegraphics[width=0.8\linewidth]{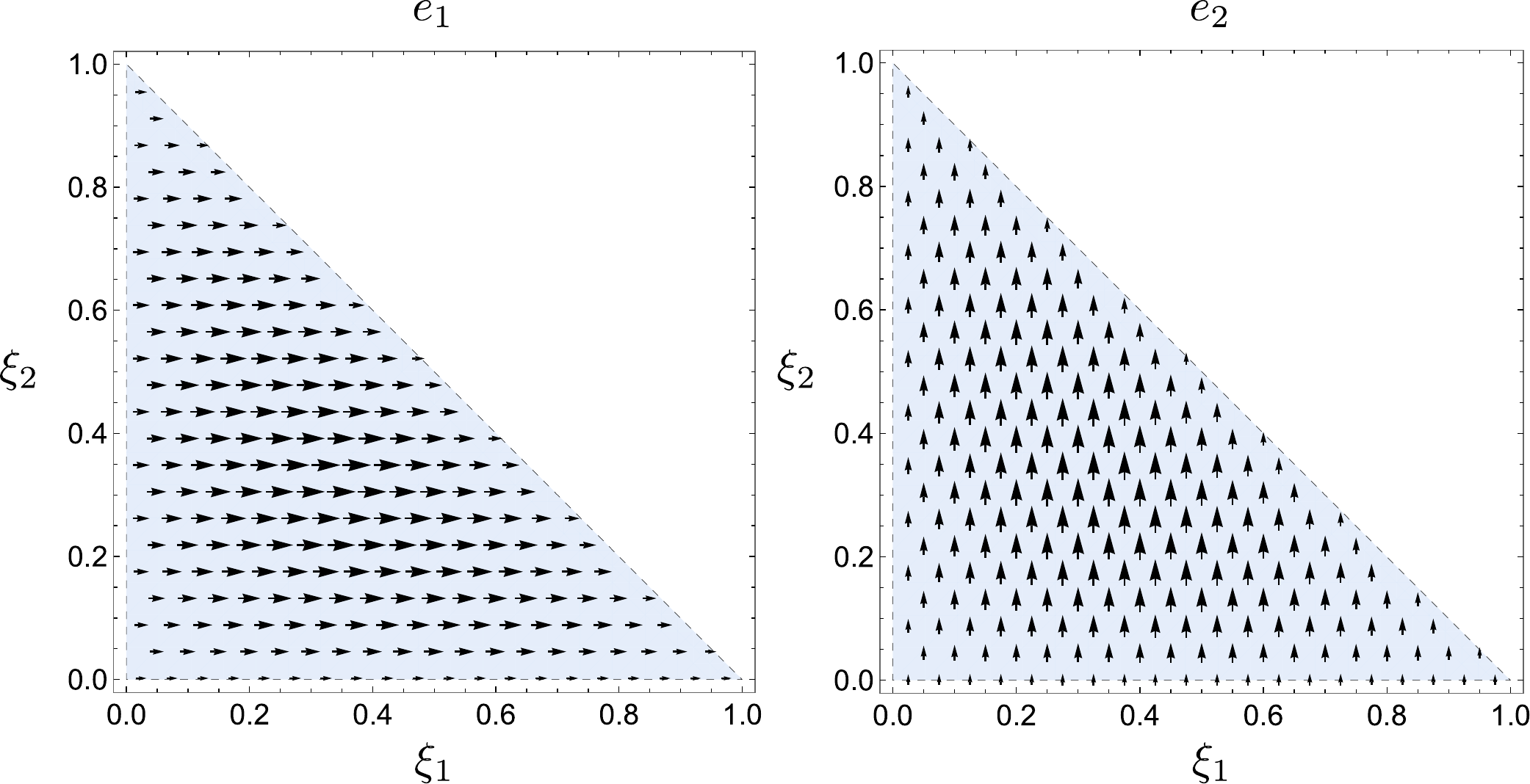}
			\caption{\normalfont An illustration of the vector fields $e_{i}$ for Example \ref{exmp:simplex_exmp_no_deception} in the $\xi$-coordinates. The arrows' sizes represent the magnitude of the vector field.}
			\label{fig:P2_exmp_vectorfields}
		\end{figure}
		We also point out that the vector fields $e_i$ do not commute. In fact, direct computation gives
		\begin{align*}
			[e_1,e_2](x) &= x_1x_2x_3(x_1(x_3-x_2)\partial_1+x_2(x_1-x_3)\partial_2\\
			&+x_3(x_2-x_1)\partial_3),
		\end{align*}
		which, generically, is non-vanishing on $\mathcal{U}$. \hfill$\Box$
	\end{exmp}

	\vspace{0.1cm}
	We will be mainly interested in the behavior of the dynamical system \eqref{eq:liebra_system} within the open sub-manifold $\mathcal{U}$. As such, Assumption \ref{asmp:indpenendent_families} allows us to endow system \eqref{eq:liebra_system} with a natural geometric interpretation. In particular, our first result shows that the flow of \eqref{eq:liebra_system} is actually a \emph{Riemannian pseudo-gradient flow on $\mathcal{M}$.} 

	\vspace{0.1cm}
	\begin{thm}\label{thm:pseduo_grad}
		Let Assumption \ref{asmp:indpenendent_families} be satisfied. Then, there exists a Riemannian metric $g$ on the sub-manifold $\mathcal{U}\subset\mathcal{M}$ such that system \eqref{eq:liebra_system} takes the form
		\begin{align}\label{eq:projected_grad_form}
			\dot{x}&= \sum_{i=1}^N \kappa_{i} \grad_{\mathcal{E}_i}(J_i)(x),
		\end{align}
		where $\grad_{\mathcal{E}_i}(J_i)$ is the vector field obtained through orthogonal projection of $\grad(J_i)$ with respect to $g$ onto $\mathcal{E}_i$. \hfill $\Box$
	\end{thm}

	\vspace{0.1cm}
	Since system \eqref{eq:projected_grad_form} generalizes the standard 
	gradient-play dynamics \cite{li2025passivity,shamma2005dynamic,ratliff2016characterization,mazumdar2020gradient}, we refer to \eqref{eq:liebra_system} as a \emph{non-holonomic gradient play}: each agent ascends its own payoff, but only along the state-dependent, and generically non-commuting, directions available to it. In the special case $\Man=\Real^n$ with $f_{i,j}=\partial_{(i,j)}$ in the standard coordinates, system \eqref{eq:projected_grad_form} reduces to the classical gradient play dynamics studied in \cite{ratliff2016characterization,mazumdar2020gradient}.

	\vspace{0.1cm}
	\begin{rem}
		The proof of Theorem \ref{thm:pseduo_grad} is constructive in the sense that it explicitly defines the Riemannian metric on the sub-manifold $\mathcal{U}$ for which the conclusions of the Theorem hold. Moreover, if the sub-manifold $\mathcal{U}$ contains any equilibria of the system \eqref{eq:liebra_system}, then, thanks to Assumption \ref{asmp:indpenendent_families}, the condition \eqref{eq:eqbm_coupled} simplifies to the natural geometric condition
		\begin{align}\label{eq:eqbm_uncoupled}
			\grad_{\mathcal{E}_i}(J_i)\big|_{x^\star} &= 0, & i&\in\{1,\dots,N\}.
		\end{align}
		%
		%
		Note that, unlike \eqref{eq:eqbm_coupled}, the condition \eqref{eq:eqbm_uncoupled} is independent of the gains $\kappa_{i}$, i.e., no agent can change the equilibria of \eqref{eq:liebra_system} in the sub-manifold $\mathcal{U}$ by amplifying their gain $\kappa_{i}$, which implies that such equilibria are solely determined by the payoffs $J_i$ and the family of vector fields $\cup_{i=1}^N\{f_{i,j}\}_{j=1}^{d_i}$,
		i.e., by the properties of the environment (cf. Figure \ref{fig:env_model}). Therefore, in the sequel we may take $\kappa_i=1$ for all $i$, which is equivalent to absorbing the constants $\kappa_i$ into the payoff functions, without affecting the equilibria of \eqref{eq:liebra_system} in the sub-manifold $\mathcal{U}$. \hfill $\square$
	\end{rem}

	\vspace{-0.2cm}
	\subsection{Stability via Intrinsic Jacobian Characterization}
	We now proceed to study the stability properties of the equilibria of \eqref{eq:liebra_system}. Specifically, we would like to characterize sufficient conditions under which an equilibrium point, i.e., a point satisfying \eqref{eq:eqbm_uncoupled}, is \emph{isolated} and locally exponentially stable. This is achieved using the \emph{intrinsic Jacobian}.
	\subsubsection{Intrinsic Jacobian and Exponential Stability}
	The intrinsic Jacobian of the vector field defining \eqref{eq:liebra_system} is the $(1,1)$-tensor $A:\mathfrak{X}(\mathcal{U})\rightarrow\mathfrak{X}(\mathcal{U})$ defined by
	\begin{align}\label{eq:linearization}
		A[X]:= \nabla_X\left(\sum_{i=1}^N \grad_{\mathcal{E}_i}(J_i)\right),
	\end{align}
	for any $X\in \mathfrak{X}(\mathcal{U})$, where we utilize the linearity of the Levi-Civita connection.
	The local properties of  $x^\star\in\mathcal{U}$ satisfying \eqref{eq:eqbm_uncoupled} are dictated by the \emph{eigenvalues} of the linear map $A|_{x^\star}:T_{x^\star}\mathcal{U}\ni v \rightarrow A|_{x^\star}[v]\in T_{x^\star}\mathcal{U}$, i.e. the restriction of the Jacobian defined in \eqref{eq:linearization} to the tangent space at the point $x^\star$. We formalize this fact in the following proposition, whose proof is classical and therefore omitted.

	\vspace{0.1cm}
	\begin{prop}\label{prop:linearization_stability}
		Let Assumption \ref{asmp:indpenendent_families} be satisfied for some $\mathcal{U}\subset\mathcal{M}$ and let $x^\star\in\mathcal{U}$ be such that \eqref{eq:eqbm_uncoupled} holds. Then, the following statements are true:
		\begin{enumerate}
			\item The point $x^\star$ is an isolated equilibrium for system \eqref{eq:liebra_system} if the map $A|_{x^\star}$ is an automorphism of $T_{x^\star}\mathcal{U}$.
			\item The point $x^\star$ is locally exponentially stable for system \eqref{eq:liebra_system} if and only if the map $A|_{x^\star}$ is Hurwitz.  \hfill $\Box$
		\end{enumerate}
	\end{prop}

	\vspace{0.1cm}
	\begin{rem}\label{remNash}
		Condition \eqref{eq:eqbm_uncoupled} resembles the first-order stationarity conditions for Nash equilibria. However, at this stage, its solutions cannot, in general, be identified as Nash equilibria. Indeed, \eqref{eq:eqbm_uncoupled} only guarantees stationarity of each payoff along the corresponding actuation distribution, while the Hurwitz property of \eqref{eq:linearization} characterizes the local exponential stability of the coupled gradient-play dynamics rather than player-wise optimality. A Nash interpretation requires additional integrability assumptions on the distributions, together with suitable player-wise second-order conditions. \hfill $\square$
	\end{rem}
	
	\vspace{0.1cm}
	\begin{exmp}[Identical Interest Games]\label{exampleidenticalgame}
		Consider the case where
		$J_1=J_2=\cdots=J_N = J,$
		i.e., the \emph{identical interest} game in which all agents share the same payoff \cite{monderer1996fictitious}. In that case, system \eqref{eq:liebra_system} is, essentially, an intrinsic gradient flow\footnote{An intrinsic gradient flow is a gradient-based dynamical system defined directly on a manifold using its geometric structure, without relying on coordinates or an embedding into a Euclidean space.}. In addition, the equilibrium condition \eqref{eq:eqbm_uncoupled} and the operator \eqref{eq:linearization} simplify to
		\begin{align*}
			\grad(J)|_{x^\star} &= 0, & A[X] &= \nabla_X \grad(J),
		\end{align*}
		where, since $\{P_{\mathcal{E}_i}\}_{i=1}^N$ are a complete set of projection operators, we used the property that $\sum_{i=1}^N P_{\mathcal{E}_i} = I$. Applying Proposition \ref{prop:linearization_stability} to this special case, we recover the standard conditions of optimization in the geometric setting. In particular, item 2) in Proposition \ref{prop:linearization_stability} is equivalent to the condition that the Hessian is negative definite at $x^\star$. \hfill $\square$
	\end{exmp}
	Given the discussion in Example \ref{exampleidenticalgame}, it is tempting to extend the intuition provided by the identical interest case to the general case. However, as will be shown in the sequel, the curvature induced by the geometry of the bundles $\mathcal{E}_i$ contributes in a non-trivial way to the Jacobian \eqref{eq:linearization}. Indeed, recalling that $\grad_{\mathcal{E}_i}(J_i):=P_{\mathcal{E}_i}[\grad(J_i)]$, where the $(1,1)$-tensor $P_{\mathcal{E}_i}$ is the orthogonal projection operator with respect to the metric onto the distribution $\mathcal{E}_i$, and utilizing the product rule of the Levi-Civita connection, we obtain
	\begin{equation}\label{eq:pseudo-hessian}
		\begin{aligned}
			A = \sum_{i=1}^N P_{\mathcal{E}_i}[\nabla \grad(J_i)]+ (\nabla P_{\mathcal{E}_i}) [\grad(J_i)].
		\end{aligned}
	\end{equation}
	From \eqref{eq:pseudo-hessian}, it is immediately clear that the curvature of the bundles $\mathcal{E}_i$ has a non-negligible contribution to the Jacobian. Indeed, the second term in \eqref{eq:pseudo-hessian} does not vanish generically at an equilibrium point $x^\star$ since condition \eqref{eq:eqbm_uncoupled} may hold at a point $x^\star\in\mathcal{U}$ even if
	$\grad(J_i)|_{x^\star}\neq 0.$
	However, using the orthogonal splitting \eqref{eq:tangent_splitting}, we may decompose $A$ into a \emph{block} form by defining the linear maps
	\begin{align*}
		A_{ij}[\cdot]=\sum_{k=1}^N P_{\mathcal{E}_i} \Big[P_{\mathcal{E}_k} &[\nabla_{P_{\mathcal{E}_j}[\cdot]}\grad(J_k)] \\
		&+ (\nabla_{P_{\mathcal{E}_j}[\cdot]} P_{\mathcal{E}_k})[\grad(J_k)]\Big].
	\end{align*}
	Since $\mathcal{E}_{i}$ are orthogonal, it follows that $P_{\mathcal{E}_i}P_{\mathcal{E}_k} = \delta_{ik} P_{\mathcal{E}_k}$, and, therefore, the tensor fields $A_{ij}$ simplify to
	\begin{subequations}\label{eq:jacobian_block_form}
		\begin{align}\label{eq:jacobian_block_form_A}
			A_{ij}[\cdot]&=P_{\mathcal{E}_i}[\nabla_{P_{\mathcal{E}_j}[\cdot]}\grad(J_i)] + C_{ij}[\cdot],
		\end{align}
	\end{subequations}
	where $C_{ij}$ are the tensor fields defined by $C_{ij}[\cdot]:= \sum_{k=1}^N P_{\mathcal{E}_i}[(\nabla_{P_{\mathcal{E}_j}[\cdot]} P_{\mathcal{E}_k})[\grad(J_k)]]$. From this block form, it is possible to derive a sufficient condition to ensure that the real parts of the eigenvalues of the linear map $A|_{x^\star}$ are strictly negative and, therefore, that $x^\star$ is both isolated and locally exponentially stable (via Proposition \ref{prop:linearization_stability}). The following result formalizes this:
	\begin{prop}
		\label{prop:weighted-block-gershgorin-riem}
		Let $x^\star\in\Dom$ satisfy the first-order condition \eqref{eq:eqbm_uncoupled} and let $A_{ij}$ be the blocks defined in \eqref{eq:jacobian_block_form}. For each pair \((i,j)\), define \(\|A_{ij}\|_{g}\) as the operator norm induced by \(g\), i.e.,
		\[
		\|A_{ij}\|_{g}
		:=
		\sup_{v\in \mathcal E_j(x^\star)\setminus\{0\}}
		\frac{\|A_{ij}v\|_{g}}{\|v\|_{g}},
		\]
		and, for a linear map $M$ on $\mathcal{E}_i(x^\star)$, let $\sigma_{\min}^g(M):=\inf_{v\in \mathcal E_i(x^\star)\setminus\{0\}}\|Mv\|_{g}/\|v\|_{g}$. Suppose that there exist positive constants \(\chi_i>0\) (\(i=1,\dots,N\)) such that, for each $i$,
		\begin{equation}
			\label{eq:wbg-riem}
			\inf_{\Re(s)\ge 0}\sigma_{\min}^g(A_{ii}|_{x^\star}-sI_i)
			\;>\;
			\sum_{j\neq i}\|A_{ij}|_{x^\star}\|_{g}\,\frac{\chi_j}{\chi_i},
		\end{equation}
		where $I_i$ denotes the identity map on $\mathcal{E}_i(x^\star)$. Then \(A|_{x^\star}\) is an automorphism and Hurwitz. \hfill $\Box$
	\end{prop}

	\vspace{0.1cm}
	\subsubsection{Jacobian Structure of Nonholonomic Gradient Play} In the remainder, we impose the following assumption:

	\vspace{0.1cm}
	\begin{asmp}\label{asmp:independent_vectorfields}
		For each $i \in \{1, \ldots, N\}$, the family $\{f_{i,j}\}_{j=1}^{d_i}$ is linearly independent at every $x \in \Dom$.\hfill $\Box$
	\end{asmp}
	\vspace{0.1cm}
	\begin{rem}
		Under Assumption \ref{asmp:independent_vectorfields}, the Riemannian metric $g$ on $\Dom$ constructed in the proof of Theorem \ref{thm:pseduo_grad} makes the family $\{f_{i,j} : 1 \le i \le N,\; 1 \le j \le d_i\}$ a global orthonormal frame on $\Dom$ under the metric $g$. Note, however, that this Assumption is not too restrictive and is only made for convenience. Indeed, by restricting the submanifold $\mathcal{U}\subset\mathcal{M}$ if needed, we may always select a set of independent vector fields that span each distribution $\mathcal{E}_i$.\hfill $\square$
	\end{rem}
	Next, we proceed to define $\{\vartheta^{i,j}\}$ as the co-frame dual to $\{f_{i,j} : 1 \le i \le N,\; 1 \le j \le d_i\}$, i.e., the unique set of differential $1$-forms on $\mathcal{U}$ that satisfy
	\begin{equation}
		\vartheta^{i,j}[f_{k,\ell}] = \delta_{ik}\,\delta_{j\ell}. \label{eq:dual}
	\end{equation}
	With this definition, the metric and the orthogonal projection operators onto the distributions $\Ei_i$ take the form
	\begin{align}\label{eq:metric-proj}
		g(v,w) &= \sum_{i,j} \vartheta^{i,j}[v]\, \vartheta^{i,j}[w], &
		P_{\Ei_i} &= \sum_{j=1}^{d_i} \vartheta^{i,j} \otimes f_{i,j}.
	\end{align}
	In addition, for any smooth function $J$,
	\begin{align}\label{eq:grad}
		\grad(J) &= \sum_{i=1}^N \grad_{\Ei_i}(J), &
		\grad_{\Ei_i}(J) &= \sum_{j=1}^{d_i} dJ[f_{i,j}]\, f_{i,j}.
	\end{align}

	\vspace{-0.5cm}
	We then have the following Theorem.
	\medskip
	\begin{thm}
		\label{thm:jacobian_simplification}
		Suppose that Assumptions \ref{asmp:indpenendent_families}-\ref{asmp:independent_vectorfields} hold. With respect to the frame $\{f_{i,j}\}$, the linear map $A|_{x^\star}$ admits the block matrix representation $A|_{x^\star} = \big[\,A_{ij}|_{x^\star}\,\big]_{i,j=1}^{N}$, $A_{ij}|_{x^\star} \in \Real^{d_i \times d_j}$, where the entries of the matrices $A_{ij}|_{x^\star}$ are given by
		\begin{equation}
			\label{eq:block-entries}
			\begin{aligned}
				\big(A_{ij}&|_{x^\star}\big)_{p,q} = d\big(dJ_i[f_{i,p}]\big)[f_{j,q}]\Big|_{x^\star}\\
				&= \hess(J_i)\big[f_{i,p}, f_{j,q}\big]\Big|_{x^\star} + dJ_i\big[\nabla_{f_{j,q}} f_{i,p}\big]\Big|_{x^\star},
			\end{aligned}
		\end{equation}
		for $p \in \{1,\ldots,d_i\}$ and $q \in \{1,\ldots,d_j\}$, and where $\hess(J_i)$ denotes the Hessian of $J_i$ with respect to $g$. \hfill $\Box$
	\end{thm}
	\vspace{0.2cm}
	\begin{rem}[Geometric Structure of the Linearization]\label{rem:connection_terms}
		Formula (20) decomposes each block of the Jacobian into a
		Hessian term and a connection term $dJ_i[\nabla_{f_{j,q}}f_{i,p}]$.
		The latter captures how variations in the actuation frame
		interact with the payoff gradient. Indeed, at equilibrium,
		each agent's payoff is stationary only along its own
		actuation directions, not necessarily along the entire
		tangent space. Thus, $\operatorname{grad}(J_i)|_{x^\star}$
		may be nonzero, allowing the connection term to persist
		even at equilibrium. Three special cases illustrate the role of the geometry.
		First, if the inter-block connection coefficients vanish,
		i.e., $\nabla_{f_{j,q}}f_{i,p}=0$ whenever $i\neq j$,
		then the off-diagonal blocks reduce to
		\[
		(A_{ij}|_{x^\star})_{p,q}
		=
		\operatorname{Hess}(J_i)
		[f_{i,p},f_{j,q}]|_{x^\star},
		\qquad i\neq j,
		\]
		while the diagonal blocks retain the connection terms
		associated with the intra-block geometry. 
		Conversely, if the intra-block connection coefficients
		vanish, i.e., $\nabla_{f_{i,q}}f_{i,p}=0$ for all $i$,
		then the diagonal blocks reduce to pure Hessian blocks,
		while the off-diagonal blocks retain the connection terms
		associated with the inter-block geometry.
		Finally, if the frame is holonomic, i.e.,
		$[f_{i,p},f_{j,q}]=0$ for all pairs, then it defines
		a coordinate frame in a suitable local chart, wherein
		the metric constructed in Theorem 1 is Euclidean.
		Consequently, all connection terms vanish, and (20)
		reduces to
		\[
		(A_{ij}|_{x^\star})_{p,q}
		=
		\partial^2_{(j,q),(i,p)}J_i|_{x^\star},
		\]
		recovering the characterization in [9].
		In general, however, the connection terms must be
		accounted for when analyzing the local stability of
		nonholonomic gradient-play dynamics. \hfill $\square$
	\end{rem}
	\subsubsection{Player-wise Involutivity of Control Distributions}
	We now consider additional assumptions on the family of vector fields and investigate the resulting simplifications on the geometric structures introduced earlier. We begin with the case in which each agent's control distribution is involutive.
	\begin{asmp}[Player-wise Involutivity]\label{asmp:involutive_distributions}
		Suppose that, for each $i\in\{1,\dots,N\}$, $[f_{i,p},f_{i,q}]\in \mathrm{span}\{f_{i,r}\}_{r=1}^{d_i}$, for all $p,q\in\{1,\dots,d_i\}$. \hfill $\square$
	\end{asmp}
	\begin{rem}
		In words, Assumption \ref{asmp:involutive_distributions} asks that each agent cannot directly interfere in the directions influenced by other agents by exciting Lie Brackets of its own vector fields, which is a natural assumption in practical applications. Note that Assumption \ref{asmp:involutive_distributions} holds trivially in Examples \ref{exmp:sphere_exmp_no_deception} and \ref{exmp:simplex_exmp_no_deception}, since each distribution therein is of rank one.  \hfill $\square$
	\end{rem}
	
	Under Assumptions \ref{asmp:indpenendent_families} and \ref{asmp:involutive_distributions}, each distribution $\mathcal{E}_i$ is integrable. More precisely, by the Frobenius Theorem \cite[Theorem 19.12]{lee2012smooth}, through every $x\in\mathcal{U}$ there passes a unique maximal connected integral manifold of $\mathcal{E}_i$, which we denote by $\mathcal{U}_i(x)$ and call the \emph{leaf} of $\mathcal{E}_i$ through $x$, satisfying
	\begin{equation}\label{leaf}
		T_y\,\mathcal{U}_i(x)=\mathcal{E}_i(y),
		\qquad
		y\in\mathcal{U}_i(x).
	\end{equation}
	Note, however, that the existence of these foliations is not enough to conclude that $\mathcal{U}$ admits a product structure whose factors are integral manifolds of the distributions $\mathcal{E}_i$. 
	Indeed, each distribution $\mathcal{E}_i$ can be $1$-dimensional, which is trivially involutive, and yet the manifold $\mathcal{U}$ does not admit a product structure that respects the splitting in Assumption \ref{asmp:indpenendent_families} due to the fact that vector fields belonging to different distributions may not commute. 
	Nevertheless, Assumption \ref{asmp:involutive_distributions} allows us to simplify the expressions in \eqref{eq:block-entries}. In particular, under Assumption \ref{asmp:involutive_distributions}, each leaf $\mathcal{U}_i(x)\subset\mathcal{U}$ is a \emph{weakly embedded} submanifold \cite[Theorem 19.17]{lee2012smooth} of the Riemannian manifold $\mathcal{U}$ and, when equipped with the ambient metric, is isometrically weakly embedded. Consequently, restrictions of smooth functions to $\mathcal{U}_i(x)$ are also smooth functions. Whenever a point $x^\star\in\mathcal{U}$ is fixed, we abbreviate $\mathcal{U}_i^\star:=\mathcal{U}_i(x^\star)$. Introducing the notation $J_i|_{\mathcal{U}_i^\star}$ to denote the restriction of the smooth function $J_i$ to the leaf $\mathcal{U}_i^\star$, as well as $g^{i}$, $\nabla^{i}$, to denote the Riemannian metric on $\mathcal{U}_i^\star$ induced by the ambient metric $g$, and its associated Levi-Civita connection, respectively, we have the following result.
	\begin{prop}\label{prop:involutive_dists}
		Suppose that Assumptions \ref{asmp:indpenendent_families}-\ref{asmp:involutive_distributions} hold, and that $x^\star\in\mathcal{U}$ satisfies \eqref{eq:eqbm_uncoupled}. Then, with respect to the frame $\{f_{i,j}\}$, the diagonal blocks in \eqref{eq:block-entries} are given by
		\begin{equation}\label{hessianexpression}
			(A_{ii}|_{x^\star})_{p,q} = \hess^{i}(J_i|_{\mathcal{U}_i^\star})[f_{i,p},{f_{i,q}}]|_{x^\star},
		\end{equation}
		where $\hess^{i}$ is the intrinsic Hessian of the restriction $J_i|_{\mathcal{U}_i^\star}$, computed with respect to the induced metric $g^{i}$. \hfill $\square$
	\end{prop}
	\begin{rem}\label{rem:deception_breaks_involutivity}
		A major simplification provided by Proposition \ref{prop:involutive_dists} is the fact that the matrix $A_{ii}|_{x^\star}$ is now symmetric, since the Hessian is a symmetric operator. \hfill $\square$
	\end{rem}
	\begin{rem}
		We note that Assumption \ref{asmp:involutive_distributions} will no longer hold, in general, if deception is involved. Indeed, as shown later in Proposition \ref{prop:tilted_directions_independent}, the effect of deception in the presence of information asymmetry is to tilt the distributions $\mathcal{E}_i$ controlled by the oblivious agents along the distributions controlled by the deceiver. Therefore, as soon as deception is involved, the new distributions $\tilde{\mathcal{E}}_i$ may no longer be involutive in the general case. This is a major motivation behind considering the general geometric analysis presented in this section, which allows us to study the stability of deception in Section \ref{sec:deception_stability}. \hfill $\square$
	\end{rem}
	\vspace{0.1cm}
	\subsubsection{Intra-Agent Commutativity of Vector Fields}
	Next, we consider the following Assumption.
	\begin{asmp}[Cross-player Commutativity]\label{asmp:commuting_vectorfields}
		For all $i,j\in\{1,\dots,N\}$ with $i\neq j$,
		$$[f_{i,p},f_{j,q}]=0,$$
		for all $p\in\{1,\dots,d_i\}$ and all $q\in\{1,\dots,d_j\}$.
	\end{asmp}
	\begin{rem}
		In words, Assumption \ref{asmp:commuting_vectorfields} requires that the vector fields controlled by each agent commute with those controlled by every other agent. Note that Assumption \ref{asmp:commuting_vectorfields} fails in both Examples \ref{exmp:sphere_exmp_no_deception} and \ref{exmp:simplex_exmp_no_deception}, in view of the non-vanishing commutators computed therein, which illustrates that the product-structure setting considered above is genuinely more restrictive. \hfill $\square$
	\end{rem}

	Under Assumptions \ref{asmp:indpenendent_families}-\ref{asmp:commuting_vectorfields}, Frobenius Theorem can be applied to show the existence of a local
	diffeomorphism $\phi : \mathcal{U} \to \mathcal{U}_1 \times \cdots \times \mathcal{U}_N$,  where each $\mathcal{U}_i$ is a smooth $d_i$-dimensional Riemannian manifold and $\phi_* \mathcal{E}_i = T\mathcal{U}_i$. In these coordinates, the leaf $\mathcal{U}_i(x)$ is identified with the product slice through $\phi(x)$. In particular,
	the Riemannian metric $g$ from Theorem \ref{thm:pseduo_grad} takes now a block-diagonal form $g = g^1 \oplus \cdots \oplus g^N$ in these coordinates. The following proposition establishes the key geometric consequence of the combined assumptions: the Levi-Civita connection of $g$ has no inter-agent components.
	\vspace{0.1cm}
	\begin{prop}\label{prop:commuting_vectorfields}
		Let Assumptions \ref{asmp:indpenendent_families}-\ref{asmp:commuting_vectorfields} be satisfied. Then,
		\begin{equation}\label{eq:vanishing_connection}
			\nabla_{f_{j,q}} f_{i,p} = 0
			\quad \text{on } \mathcal{U},
		\end{equation}
		for all $i \neq j$. \hfill $\Box$
	\end{prop}
	As a consequence of Proposition \ref{prop:commuting_vectorfields}, the intrinsic Jacobian from Theorem \ref{thm:jacobian_simplification} simplifies
	completely.
	\begin{cor}
		\label{cor:product_jacobian}
		Let the hypotheses of Proposition \ref{prop:commuting_vectorfields} hold and let $x^\star \in \mathcal{U}$ satisfy \eqref{eq:eqbm_uncoupled}. Then, the blocks of the intrinsic Jacobian $A|_{x^\star}$ are
		given by
		\begin{equation}
			\label{eq:product_jacobian}
			\left(A_{ij}\big|_{x^\star}\right)_{p,q}
			= \begin{cases}
				\hess^i(J_i|_{\mathcal{U}_i^\star})\!\left[f_{i,p},\, f_{i,q}\right]
				\big|_{x^\star},
				& i = j, \\[6pt]
				\hess(J_i)\!\left[f_{i,p},\, f_{j,q}\right]
				\big|_{x^\star},
				& i \neq j.
			\end{cases}
		\end{equation}
		
		\phantom{.}\hfill $\Box$
	\end{cor}
	\begin{proof}
		For diagonal blocks, the expression in \eqref{eq:product_jacobian} follows
		from Proposition \ref{prop:involutive_dists}, which
		requires only Assumptions \ref{asmp:indpenendent_families}-\ref{asmp:involutive_distributions}.
		For
		off-diagonal blocks with $i \neq j$, substituting
		\eqref{eq:vanishing_connection} into the second equality of \eqref{eq:block-entries} immediately gives $(A_{ij}|_{x^\star})_{p,q} =
		\hess(J_i)[f_{i,p}, f_{j,q}]|_{x^\star}$.
	\end{proof}
	
	\vspace{-0.2cm}
	\subsection{Leafwise Nash Equilibrium: Regularity and Stability}
	In this subsection, we establish conditions under which the equilibria of the
	nonholonomic gradient-play dynamics \eqref{eq:liebra_system} correspond to
	Nash equilibria of the underlying game. To do this, note that, when each player's distribution $\mathcal{E}_i$ is involutive, i.e. when Assumption~\ref{asmp:involutive_distributions} holds, the leaf $\mathcal{U}_i(x)$ of $\mathcal{E}_i$ through $x$, defined in \eqref{leaf}, characterizes the set of states accessible to player $i$ through unilateral motions along its actuation distribution. This geometric notion naturally leads to the concept of \emph{leafwise Nash equilibria}, formalized below:
	\begin{defn}\label{lne}
		Suppose that Assumption~\ref{asmp:involutive_distributions} holds, and let $\mathcal{U}_i(x)$ be the leaf of $\mathcal{E}_i$ through $x$, as in \eqref{leaf}.  A point $x^\star\in\mathcal{U}$ is said to be a strict local leafwise Nash equilibrium if, for every $i\in\{1,\dots,N\}$, there exists a
		neighborhood $\mathcal{O}_i$ of $x^\star$ in
		$\mathcal{U}_i(x^\star)$ such that $J_i(x^\star)>J_i(y)$, for all $y\in\mathcal{O}_i\setminus\{x^\star\}$. \hfill $\square$
	\end{defn}
	\begin{figure}[t]
		\centering
		\includegraphics[width=0.9\linewidth]{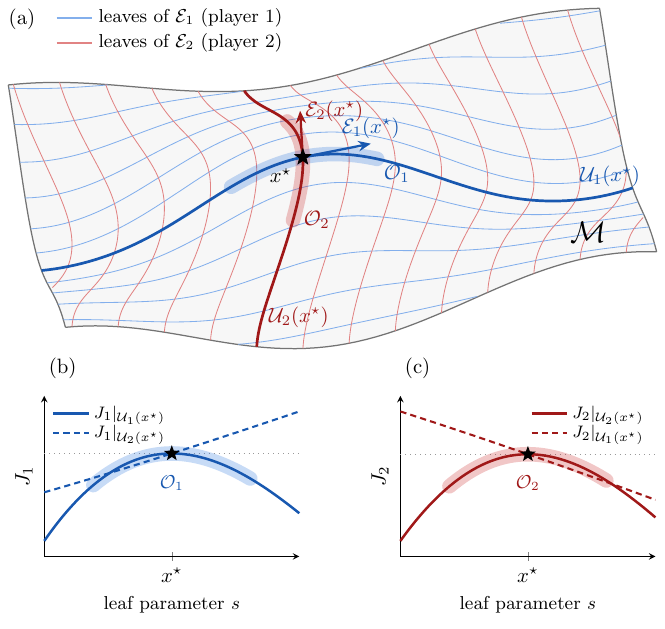}
		\caption{Leafwise Nash equilibrium for two players on a surface $\mathcal{M}$. (a)~A unilateral deviation of player~$i$ from $x^\star$ is confined to the leaf $\mathcal{U}_i(x^\star)$ through $x^\star$. The shaded arcs indicate the neighborhoods $\mathcal{O}_i\subset\mathcal{U}_i(x^\star)$ of $x^\star$. (b),~(c)~The payoffs along the leaves through $x^\star$, parametrized by the leaf parameter $s$ with $s=0$ at $x^\star$. The restriction of $J_i$ to its own leaf $\mathcal{U}_i(x^\star)$ (solid) attains a strict local maximum at $x^\star$ on $\mathcal{O}_i$. The restriction of $J_i$ to the other leaf $\mathcal{U}_j(x^\star)$, $j\neq i$, (dashed) is not stationary at $x^\star$, i.e., $x^\star$ is not a critical point of $J_i$ on $\mathcal{M}$ in general.}
		\label{fig:LWNE}
	\end{figure}

	\vspace{0.1cm}
	
	A geometric illustration of leafwise Nash equilibria is shown in Figure \ref{fig:LWNE}. The following proposition shows that, under the given assumptions, the equilibria of \eqref{eq:liebra_system} are precisely leafwise Nash equilibria.
	%
	
	\begin{prop}\label{propLNE}
		Suppose that Assumptions~\ref{asmp:indpenendent_families} and~\ref{asmp:involutive_distributions} hold, and that $x^\star\in\mathcal{U}$ satisfies the first-order condition \eqref{eq:eqbm_uncoupled} and the following second-order condition
		\begin{equation}\label{eq:leafwise_second_order}
			\hess^{i}
			\big(J_i|_{\mathcal{U}_i^\star}\big)\big|_{x^\star}
			\prec 0,~~\forall~i\in\{1,\dots,N\}.
		\end{equation}
		Then, $x^\star$ is a strict
		local leafwise Nash equilibrium. \hfill $\square$
	\end{prop}

	\vspace{0.1cm}
	Under Assumptions~\ref{asmp:indpenendent_families}-\ref{asmp:involutive_distributions}, Proposition~\ref{prop:involutive_dists} showed, via \eqref{hessianexpression}, that $A_{ii}|_{x^\star}$ is the matrix representation of $\hess^{i}(J_i|_{\mathcal{U}_i^\star})\big|_{x^\star}$ with respect to the orthonormal frame
	$\{f_{i,p}(x^\star)\}_{p=1}^{d_i}$. Consequently, the second-order condition \eqref{eq:leafwise_second_order} can equivalently be
	written as $A_{ii}|_{x^\star}\prec 0$. We therefore obtain the
	following corollary.
	\begin{cor}
		Suppose that Assumptions~\ref{asmp:indpenendent_families}-\ref{asmp:involutive_distributions} hold, and that $x^\star\in\mathcal{U}$
		satisfies~\eqref{eq:eqbm_uncoupled}. If $A_{ii}|_{x^\star}\prec 0$, for all $i\in\{1,\dots,N\}$, then $x^\star$ is a strict local leafwise Nash equilibrium. \hfill $\square$
	\end{cor}

	Finally, we study the role of Assumption \ref{asmp:commuting_vectorfields} in the game-theoretic nature of $x^{\star}$. To do this, note that, under Assumptions~\ref{asmp:indpenendent_families}-\ref{asmp:commuting_vectorfields}, the manifold $\mathcal{U}$ admits a local product
	decomposition compatible with the players' actuation
	distributions. Namely, there exist an open neighborhood $U\subset\mathcal M$ of $x^\star$, open sets $V_i\subset\mathbb R^{d_i}$, $i\in\{1,\ldots,N\}$, and
	a diffeomorphism $\varphi:U\to V_1\times\cdots\times V_N,\,\varphi(x)=(z_1,\ldots,z_N)$, such that, for every $z\in V_1\times\cdots\times V_N$ and every
	$i\in\{1,\ldots,N\}$,
	\[
	D\varphi^{-1}(z)
	\left(
	\{0\}\times\cdots\times\mathbb R^{d_i}
	\times\cdots\times\{0\}
	\right)
	=
	\mathcal E_i\bigl(\varphi^{-1}(z)\bigr).
	\]
	Let $z^\star:=\varphi(x^\star)
	=(z_1^\star,\ldots,z_N^\star)$, and define the local coordinate representation of the payoff of
	player $i$ as  $\widehat J_i
	:=J_i\circ\varphi^{-1}$. The following definition leverages the functions $\widehat J_i$ to provide the ``usual'' game-theoretic
	interpretation of Nash equilibria, where each player $i$ changes only its $i$th coordinate, while all the other coordinates remain fixed.
	\begin{defn}\label{classicnash}
		The point $x^\star$ is said to be a \emph{strict local Nash equilibrium in the
			local product coordinates induced by
			$\{\mathcal E_i\}_{i=1}^N$} if, for every
		$i\in\{1,\ldots,N\}$, there exists a neighborhood
		$W_i\subset V_i$ of $z_i^\star$ such that
		\[
		\widehat J_i(z_i^\star,z_{-i}^\star)
		<
		\widehat J_i(z_i,z_{-i}^\star),
		\qquad
		\forall z_i\in W_i\setminus\{z_i^\star\},
		\]
		where $z_{-i}^\star
		:=
		(z_1^\star,\ldots,z_{i-1}^\star,
		z_{i+1}^\star,\ldots,z_N^\star)$.
	\end{defn}

	The following proposition is the last result of this section, and it provides sufficient conditions under which the equilibria of \eqref{eq:liebra_system} are standard Nash equilibria in product coordinates.
	\begin{prop}\label{standardNE} Suppose that Assumptions \ref{asmp:indpenendent_families}-\ref{asmp:commuting_vectorfields} hold, and let
		$x^\star\in\mathcal{U}$ satisfy the first order condition \eqref{eq:eqbm_uncoupled} and the inequality
		\begin{equation}
			A_{ii}|_{x^\star}\prec0,
			\qquad
			i\in\{1,\dots,N\},
			\label{eq:product_nash_second_order}
		\end{equation}
		where $A_{ii}$ denotes the $i$th diagonal block of the
		intrinsic Jacobian \eqref{eq:product_jacobian}. Then, $x^\star$ is a strict local Nash equilibrium in the local product coordinates induced by 
		$\{\mathcal{E}_i\}_{i=1}^N$.
		\hfill $\square$\end{prop}
	
	\vspace{0.1cm}
	\begin{rem}
		\normalfont
		In the special case where each $\mathcal{E}_i$ is spanned by coordinate vector fields in some chart,
		i.e., $f_{i,j} = \partial_{(i,j)}$, Assumptions \ref{asmp:indpenendent_families}-\ref{asmp:commuting_vectorfields} are automatically satisfied, and system \eqref{eq:liebra_system} takes the form
		\begin{align*}
			\dot{x} = \sum_{i=1}^N\sum_{j=1}^{d_i} \partial_{(i,j)}J_i(x) \partial_{(i,j)},
		\end{align*}
		in coordinates, i.e., the interaction is governed by the vector field whose components are partial derivatives of the payoff functions in local coordinates. Moreover, we compute that
		\begin{align*}
			\hess^i(J_i)[f_{i,p}, f_{i,q}] &= \partial^2_{(i,p),(i,q)} J_i, \\
			\hess(J_i)[f_{i,p}, f_{j,q}] &= \partial^2_{(j,q),(i,p)} J_i.
		\end{align*}
		In this case, \eqref{eq:product_jacobian} recovers
		the coordinate dependent formulas studied in \cite{ratliff2016characterization,mazumdar2020gradient} as a special case of Theorem \ref{thm:jacobian_simplification}. \hfill $\square$
	\end{rem}
	\begin{rem}\label{rem:N=2}
		We remark that, when $N=2$, the conclusion of Proposition~\ref{standardNE} holds without Assumption~\ref{asmp:commuting_vectorfields}, i.e., a strict local leafwise Nash equilibrium is a strict local Nash equilibrium in local product coordinates, and conversely. Indeed, under Assumptions~\ref{asmp:indpenendent_families} and~\ref{asmp:involutive_distributions}, the Frobenius Theorem applied separately to $\mathcal{E}_2$ and to $\mathcal{E}_1$ yields, near $x^\star$, submersions $z_1$ and $z_2$ with $\ker dz_1=\mathcal{E}_2$ and $\ker dz_2=\mathcal{E}_1$. Since $\mathcal{E}_1\cap\mathcal{E}_2=\{0\}$, the map $\phi=(z_1,z_2)$ is a local diffeomorphism that satisfies $\phi_*\mathcal{E}_i=T\mathcal{U}_i$ for $i\in\{1,2\}$, and the proof of Proposition~\ref{standardNE} applies without change. In particular, this is the case when $d_1=d_2=1$, as in Examples~\ref{exmp:sphere_exmp_no_deception_2} and \ref{exmp:simplex_exmp_no_deception_2}, since Assumption~\ref{asmp:involutive_distributions} holds automatically. We emphasize that this product structure is smooth but not necessarily Riemannian, and the metric $g$ need not be a product metric in the coordinates $(z_1,z_2)$. For $N\geq 3$, the same construction requires $\bigoplus_{j\neq i}\mathcal{E}_j$ to be involutive for every $i$, which may fail even if every $\mathcal{E}_i$ is involutive.
	\end{rem}
	\subsection{Examples (continued)}
	\vspace{0.1cm}
	Before we move on from the current section, we provide examples that illustrate the concepts introduced thus far.
	\vspace{0.1cm}
	\begin{exmp}[Example \ref{exmp:sphere_exmp_no_deception} cont'd]\label{exmp:sphere_exmp_no_deception_2}
		With the vector fields $\{e_i\}$ defined in Example \ref{exmp:sphere_exmp_no_deception}, we define the co-vector fields $\{\vartheta^i\}$ as
		\begin{equation*}
			\begin{aligned}
				\vartheta^i(x)[v]:=(-1)^{3-i}\Delta(x)^{-1} \, \langle v, e_{{3-i}}(x)\times x\rangle_{\mathbb{R}^3},
			\end{aligned}
		\end{equation*}
		for all $(x,v)\in T\mathcal{U}$. Then, it can be verified through direct computation that
		$\vartheta^i[e_{j}]=\delta_{ij}$,
		i.e.,  $\{\vartheta^i\}$ is the dual co-frame to the frame $\{e_{i}\}$.
		We now define a metric $g$ on the submanifold $\mathcal{U}$ by \eqref{eq:metric-proj} for any pair of vector fields $v,w$, and we consider the payoff functions $J_i(x)=2^{\frac{1}{2}}+(x_3+x_{3-i})$, which are depicted in Figure \ref{fig:S2_exmp_payoffs}.
		Direct computation gives $dJ_i[e_{i}](x)= -x_i$. Therefore, the only equilibrium point in $\mathcal{U}$ is given by $x^\star=(0,0,1)$. To verify the structural, as well as the asymptotic, stability properties of $x^\star$, we need to compute the intrinsic Jacobian $A$ defined in \eqref{eq:linearization}. To that end, we invoke Theorem \ref{thm:jacobian_simplification} and utilize the simplified expression obtained in \eqref{eq:block-entries}. It follows that $d(dJ_i[e_i])[e_j]|_{x^\star} = -\delta_{ij}$, i.e., the matrix representation of $A|_{x^\star}$ with respect to the frame $\{e_1,e_2\}$ is given by $A|_{x^\star}=-I_{2x2}$, which is clearly an automorphism of $T_{x^\star}\mathcal{U}$ and also Hurwitz. In particular, the diagonal entries are negative which implies that $x^\star$ is strict local leafwise Nash equilibrium. It follows from Proposition \ref{prop:linearization_stability} that $x^\star$ is an isolated locally exponentially stable equilibrium. \hfill $\square$
	\end{exmp}
	
	\vspace{0.1cm}
	\begin{exmp}[Example \ref{exmp:simplex_exmp_no_deception} cont'd]\label{exmp:simplex_exmp_no_deception_2}
		Next, we define the co-vector fields $\vartheta^i$ by $\vartheta^i(x)[v]:=\frac{v_i}{x_1x_2x_3}$, for any tangent vector $v=(v_1,v_2,v_3)\in T_x\mathcal{U}$. Then, it can be verified through direct computation that $\vartheta^i[e_j]=\delta_{ij}$, i.e., the co-frame $\{\vartheta^i\}$ is dual to the frame $\{e_i\}$. We now define a metric $g$ on the submanifold $\mathcal{U}$ by \eqref{eq:metric-proj}, which, automatically, makes the frame $\{e_i\}$ orthonormal with respect to $g$. In this case, we consider the payoff functions  $J_i(x):=1-\frac{1}{2}((x_{3-i}-\alpha)^2+(x_i-1+2\alpha)^2 +(x_3-\alpha)^2)$, for any $\alpha\in\mathbb{R}$, which are depicted in Figure \ref{fig:P2_exmp_payoffs} for $\alpha = 4/9$.
		Direct computation gives that
		\begin{align}\label{eq:dJ_i}
			dJ_i[e_i](x)= - x_1 x_2 x_3(x_i-x_3+3\alpha - 1),
		\end{align}
		which can be shown to imply that the only equilibrium in $\mathcal{U}$ is the point $x^\star=\frac{1}{3}(2-3\alpha,2-3\alpha,6\alpha -1)$, provided that $\alpha \in (1/6,2/3)$. To verify the structural, as well as the asymptotic, stability properties of $x^\star$, we compute the intrinsic Jacobian $A$ by invoking Theorem \ref{thm:jacobian_simplification} and using \eqref{eq:block-entries}. Indeed, substituting from \eqref{eq:dJ_i} into \eqref{eq:block-entries}, we obtain:
		\begin{equation*}
			\begin{aligned}
				d(dJ_1[e_1])[e_1] &= x_1 x_2^2 x_3 \big(x_1 \left(3 \alpha -4 x_3-1\right)\\
				&+x_3 \left(-3 \alpha +x_3+1\right)+x_1^2\big), \\
				d(dJ_1[e_1])[e_2] &= x_1^2 x_2 x_3 \big(x_2 \left(3 \alpha +x_1-2 x_3-1\right)\\
				&+x_3 \left(-3 \alpha -x_1+x_3+1\right)\big), \\
				d(dJ_1[e_2])[e_1] &= x_1 x_2^2 x_3 \big(x_1 \left(3 \alpha +x_2-2 x_3-1\right)\\
				&+x_3 \left(-3 \alpha -x_2+x_3+1\right)\big), \\
				d(dJ_1[e_2])[e_2] &= x_1^2 x_2 x_3 \big(x_2 \left(3 \alpha -4 x_3-1\right)\\
				&+x_3 \left(-3 \alpha +x_3+1\right)+x_2^2\big),
			\end{aligned}
		\end{equation*}
		\begin{figure}[t]
			\centering
			\includegraphics[width=0.8\linewidth]{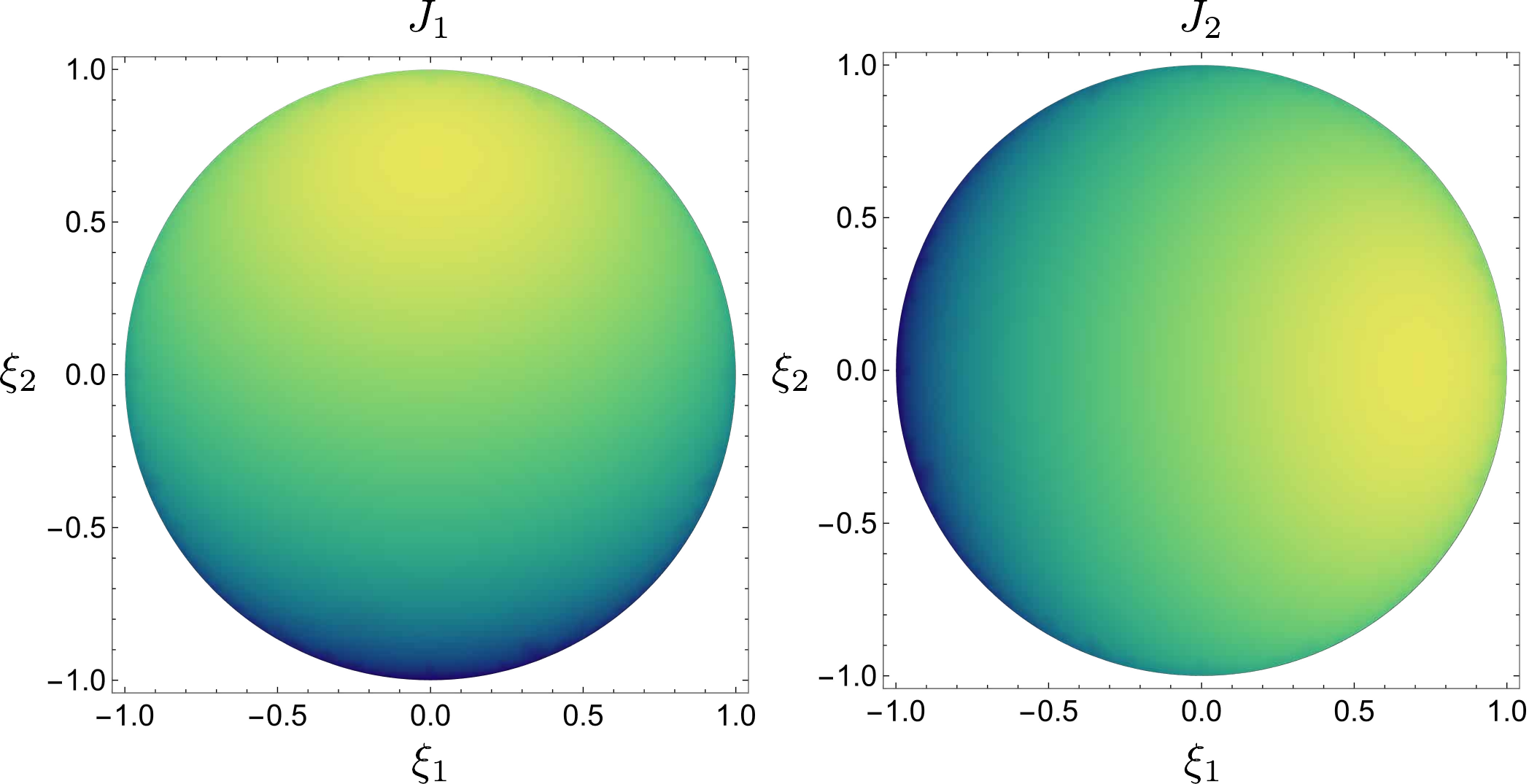}
			\caption{\normalfont An illustration of the payoffs $J_i$ for Example \ref{exmp:sphere_exmp_no_deception} in the $\xi$-coordinates. The red dots represent the maximizer of the corresponding function.}
			\label{fig:S2_exmp_payoffs}
		\end{figure}
		so that, at the point $x^\star$, we have
		\begin{align*}
			A|_{x^\star} = \frac{1}{729} (1-6 \alpha )^2 (2-3 \alpha )^4 \begin{bmatrix}
				-2 & -1 \\ -1 & -2
			\end{bmatrix},
		\end{align*}
		which is clearly an automorphism of $T_{x^\star}\mathcal{U}$ and Hurwitz. In particular, the diagonal entries are negative which implies that $x^\star$ is strict local leafwise Nash equilibrium. It also follows from Proposition \ref{prop:linearization_stability} that $x^\star$ is a non-degenerate isolated locally exponentially stable equilibrium  for the system. \hfill $\square$
	\end{exmp}
	\section{Information Asymmetry and Deception}\label{sec:deception_stability}
	As discussed in Section \ref{sec:asym}, when the assumption of information symmetry is violated, an agent with access to privileged information can bias the steady-state behavior of \eqref{eq:env_model} by manipulating only its own control inputs, without interfering with the actions or measurements of the oblivious agents. In this section, we analyze the interaction dynamics induced by the deceptive feedback law \eqref{eq:deceptive_es_law}, i.e., system \eqref{eq:liebra_sys_deception}, repeated below for convenience,
	\begin{align}
		\dot{x}= \sum_{i=1}^2\kappa_{i} \sum_{j=1}^{d_i} L_{\tilde{f}_{i,j}}(J_i)(x) \tilde{f}_{i,j}(x), \tag{\ref{eq:liebra_sys_deception}}
	\end{align}
	where, as in \eqref{eq:tilted_directions}, the vector fields $\tilde{f}_{i,j}$ are given by $\tilde{f}_{1,j}:=f_{1,j}$, and $\tilde{f}_{2,j}:= f_{2,j} + \sum_{k=1}^{d_1} \gamma_{j}^{k} f_{1,k}$.

	\subsection{Regularity and Stability under Deception}
	Similar to Section \ref{sec:no_deception_stability}, we impose Assumption \ref{asmp:indpenendent_families} on the vector fields $f_{i,j}$ and, for simplicity, we also impose Assumption \ref{asmp:independent_vectorfields}. Then, we have the following proposition.

	\vspace{0.1cm}
	\begin{prop}\label{prop:tilted_directions_independent}
		Let $\{\tilde{f}_{i,j}\}$ be given by \eqref{eq:tilted_directions} for any choice of deception parameters $\{\gamma_k^{j}\}$. Then, Assumption \ref{asmp:indpenendent_families} is satisfied on $\Dom$ with $\tilde{\Ei}_i = \mathrm{span}\{\tilde{f}_{i,j}\}_{j=1}^{d_i}$. In particular, the family $\{\tilde{f}_{i,j}\}$ remains pointwise linearly independent on $\Dom$ and $T_x\Dom = \tilde{\Ei}_1(x) \oplus \tilde{\Ei}_2(x)$, for all  $x \in \Dom$. \hfill $\Box$
	\end{prop}
	\begin{figure}[t]
		\centering
		\includegraphics[width=0.8\linewidth]{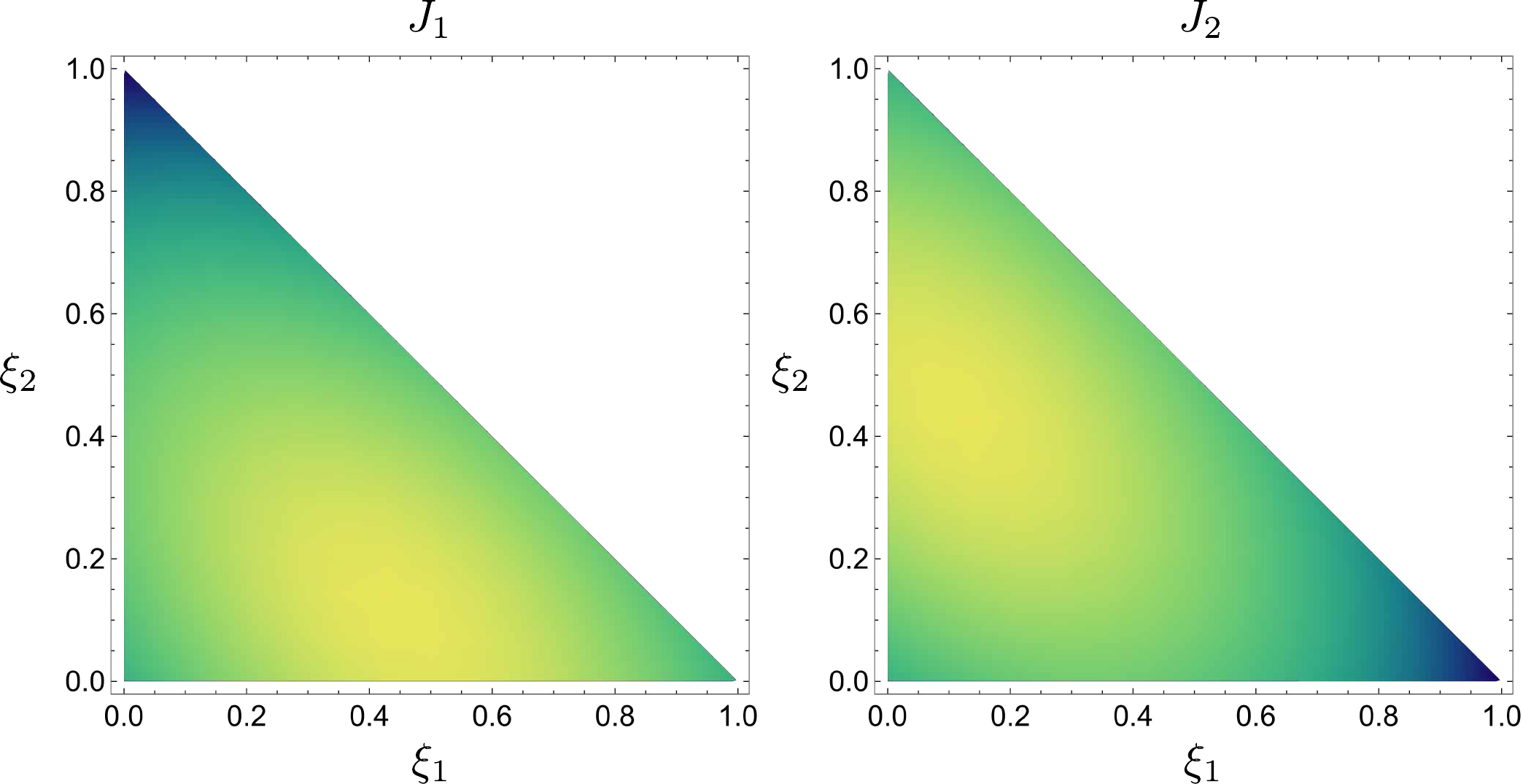}
		\caption{\normalfont An illustration of payoffs $J_i$ for Example \ref{exmp:simplex_exmp_no_deception} in the $\xi$-coordinates. The red dots represent the maximizer of the corresponding function.}
		\label{fig:P2_exmp_payoffs}
		\vspace{-0.3cm}
	\end{figure}
	\vspace{0.1cm}
	\begin{rem}[Geometric interpretation of deception]
		Even though $\tilde{\Ei}_1 = \Ei_1$ as subbundles of $T\Dom$ (since the deceiver does not alter their own directions), the victim's new bundle $\tilde{\Ei}_2$ is generically distinct from $\Ei_2$ since, from \eqref{eq:tilted_directions}, the vector field $\tilde{f}_{2,j}$ acquires components along the deceiver's directions $f_{1,k}$ with weights $\gamma_j^{k}$. This is the geometric mechanism of deception in our setting: the deceiver \emph{rotates the victim's effective actuation directions} into its own bundle without modifying its own reach. As a consequence, projected gradients on $\tilde{\Ei}_2$ are biased relative to those on $\Ei_2$, and the equilibrium of the averaged dynamics is generically shifted. Moreover, the proposition holds for \emph{all} choices of deception parameters, including arbitrarily large ones, and uniformly in $x \in \Dom$. This is a direct consequence of the block-triangular structure of $\Gamma$ in \eqref{eq:Gamma}, which decouples the invertibility analysis from the magnitude of the tilt block $\gamma$.  \phantom{.}\hfill $\square$
	\end{rem}

	\vspace{0.1cm}
	As a direct consequence of Theorem \ref{thm:pseduo_grad} and Proposition \ref{prop:tilted_directions_independent}, we have the following Corollary.
	\begin{cor}\label{cor:tilted_metric}
		There exists a Riemannian metric $\tilde{g}$ on the open sub-manifold $\mathcal{U}\subset\mathcal{M}$ such that \eqref{eq:liebra_sys_deception} takes the form
		\begin{align}\label{eq:projected_grad_form_tilted}
			\dot{x}&= \sum_{i=1}^2 \kappa_{i} \, \grad_{\tilde{\mathcal{E}}_i}(J_i)(x),
		\end{align}
		where $\grad_{\tilde{\mathcal{E}}_i}(J_i)$ is the vector field obtained through orthogonal projection of $\grad(J_i)$ with respect to $\tilde{g}$ onto $\tilde{\mathcal{E}}_i$. \hfill $\Box$
	\end{cor}

	The result of Corollary \ref{cor:tilted_metric}  allows us to, once again, take $\kappa_i=1$ and look for equilibria of system \eqref{eq:liebra_sys_deception} using the condition
	\begin{align}\label{eq:eqbm_tilted}
		\grad_{\tilde{\mathcal{E}}_i}(J_i)(x^\star) = 0, \quad i\in\{1,2\}.
	\end{align}
	In particular, since $\tilde{\mathcal{E}}_1 = \mathcal{E}_1$, this condition is unchanged from the analysis in Section \ref{sec:no_deception_stability} for $i=1$. However, unlike Section \ref{sec:no_deception_stability}, the bundle $\tilde{\mathcal{E}}_2$ is tilted along the bundle $\mathcal{E}_1$. In particular, utilizing the assumption that the vector fields $\{f_{i,j}\}$, and, consequently the vector fields $\{\tilde{f}_{i,j}\}$, are linearly independent, as well as \eqref{eq:tilted_directions}, the new equilibrium conditions take the form
	\begin{align*}
		0&=dJ_1[f_{1,j}](x), & j&\in\{1,\dots,d_1\}, \\
		0&=dJ_2[f_{2,j}](x)+\sum_{k=1}^{d_1}\gamma_{j}^{k}dJ_2[f_{1,k}](x), & j&\in\{1,\dots,d_2\},
	\end{align*}
	which implies that an equilibrium point of system \eqref{eq:liebra_sys_deception} will, in general, depend on the matrix of deception parameters $\gamma\in \mathbb{R}^{d_1\times d_2}$.
	Nevertheless, as shown in Theorem \ref{thm:deception_stability_IFT} below, the stability properties of this equilibrium point are preserved for small values of $\gamma$.
	The following theorem is the main result of this section:

	\vspace{0.1cm}
	\begin{thm}\label{thm:deception_stability_IFT}
		Suppose that $x^\star\in \mathcal{U}$ is a locally exponentially stable equilibrium point for the interaction dynamics \eqref{eq:liebra_system}. Then, there exists an open neighborhood $\Omega\subset\mathbb{R}^{d_1\times d_2}$ with $0\in\Omega$ and a smooth map $X^\star:\Omega\rightarrow\mathcal{U}$ such that $X^\star(0) = x^\star$ and, for all $\gamma\in \Omega$, the point $X^\star(\gamma)$ is a locally exponentially stable equilibrium point for the interaction dynamics \eqref{eq:liebra_sys_deception}. \hfill $\Box$
	\end{thm}

	\vspace{0.1cm}
	The following result is immediate from Theorem \ref{thm:deception_stability_IFT}.
	\begin{cor}\label{cor:sensitivity}
		Under the hypotheses of Theorem \ref{thm:deception_stability_IFT}, and under the identification of $T_{x^\star}\Dom$ with $\Real^n$ provided by the frame $\{f_{i,j}(x^\star)\}$, the map $X^\star$ satisfies
		\begin{equation}\label{eq:sensitivity}
			\frac{\partial X^\star}{\partial \gamma_j^{k}}\bigg|_{\gamma=0} = -\,dJ_2[f_{1,k}](x^\star)\;\big(A|_{x^\star}\big)^{-1}\big[f_{2,j}(x^\star)\big],
		\end{equation}
		for all $j\in\{1,\dots,d_2\}$ and $k\in\{1,\dots,d_1\}$.
		\phantom{.}\hfill $\Box$
	\end{cor}
	\medskip
	\begin{rem}
		Formula \eqref{eq:sensitivity} makes the mechanism of Proposition \ref{prop:tilted_directions_independent} quantitative: to first order, the deceiver shifts the equilibrium along the directions $(A|_{x^\star})^{-1}[f_{2,j}(x^\star)]$, weighted by the sensitivities $dJ_2[f_{1,k}](x^\star)$ of the victim's payoff along the deceiver's directions. In particular, if $dJ_2[f_{1,k}](x^\star)=0$ for all $k$, i.e., if the victim's payoff is stationary along the deceiver's directions at $x^\star$, then deception is ineffective to first order, regardless of the choice of $\gamma$. 
		\hfill $\square$
	\end{rem}
	\begin{rem}\label{rem:beyond_two_agents}
		The restriction to two agents considered in this section is made for clarity of exposition, not out of necessity. In particular, consider $N$ agents and suppose that the information-asymmetry structure is \emph{acyclic}, i.e., there exists an ordering of the agents such that agent $i$ can measure, and embed into its own inputs, the inputs of agent $\ell$ only if $i < \ell$. Then, the matrix $\Gamma$ representing the tilt, cf. \eqref{eq:Gamma}, is strictly block upper-triangular in this ordering, hence nilpotent, and $\det(I+\Gamma)=1$ irrespective of the deception parameters. Consequently, Proposition \ref{prop:tilted_directions_independent}, Theorem \ref{thm:deception_stability_IFT}, and Corollary \ref{cor:sensitivity} extend verbatim to this setting. The genuinely distinct case is that of a \emph{cyclic} information structure, e.g., two agents simultaneously attempting to deceive one another, in which case $I+\Gamma$ may become singular for certain parameter choices and the tilted family may fail to span the tangent space. \hfill $\square$
	\end{rem}

	\vspace{-0.3cm}
	\subsection{Deceptive Nash Equilibria}
	A point $x^{\star}$ satisfying Theorem \ref{thm:deception_stability_IFT} is called a \emph{deceptive Nash equilibrium} if it is a Nash equilibrium—either leafwise (c.f. Definition \ref{lne}) or in local product coordinates (c.f. Definition \ref{classicnash})—for the payoff functions $J_i$ and \emph{tilted} distributions $\{\tilde{\mathcal{E}}_i\}$. With these definitions, the following corollary follows directly from Propositions \ref{propLNE} and \ref{standardNE}:
	\begin{cor}
		For system \eqref{eq:liebra_sys_deception}, the following holds:
		\begin{enumerate}
			\item If Assumptions \ref{asmp:indpenendent_families} and \ref{asmp:involutive_distributions} hold with respect to $\{\tilde{\mathcal{E}}_i\}$ and $\{\tilde{f}_{i,k}\}$, and  $x^{\star}$ satisfies the first-order condition \eqref{eq:eqbm_tilted} and the second-order condition \eqref{eq:leafwise_second_order}, then $x^{\star}$ is a deceptive strict leafwise Nash equilibrium.
			\item If Assumptions \ref{asmp:indpenendent_families}-\ref{asmp:commuting_vectorfields} hold with respect to $\{\tilde{\mathcal{E}}_i\}$ and $\{\tilde{f}_{i,k}\}$, and
			$x^\star$ satisfies the first order condition \eqref{eq:eqbm_tilted} and the inequality \eqref{eq:product_nash_second_order}, then $x^\star$ is a deceptive strict local Nash equilibrium in the local product coordinates induced by $\{\tilde{\mathcal{E}}_i\}_{i=1}^N$.
		\end{enumerate}
	\end{cor}

	Next, we revisit the examples from Section \ref{sec:no_deception_stability} under deception.
	\begin{exmp}[Example \ref{exmp:sphere_exmp_no_deception} cont'd]\label{exmp:sphere_exmp_deception}
		We consider the case in which the $1^{\text st}$-agent implements a deceptive strategy towards the $2^{\text nd}$-agent. In this case, $\{e_i\}$ are replaced by $\tilde{e}_1(x):= e_1(x)$,  $\tilde{e}_2(x)= e_{2}(x) + \gamma e_1(x)$, where $\gamma\in\mathbb{R}$ is the deception parameter. The new vector fields are depicted in Figure \ref{fig:S2_example_vectorfields_tilted} for $\gamma = \frac{1}{2}$.
		\begin{figure}[t]
			\centering
			\includegraphics[width=0.8\linewidth]{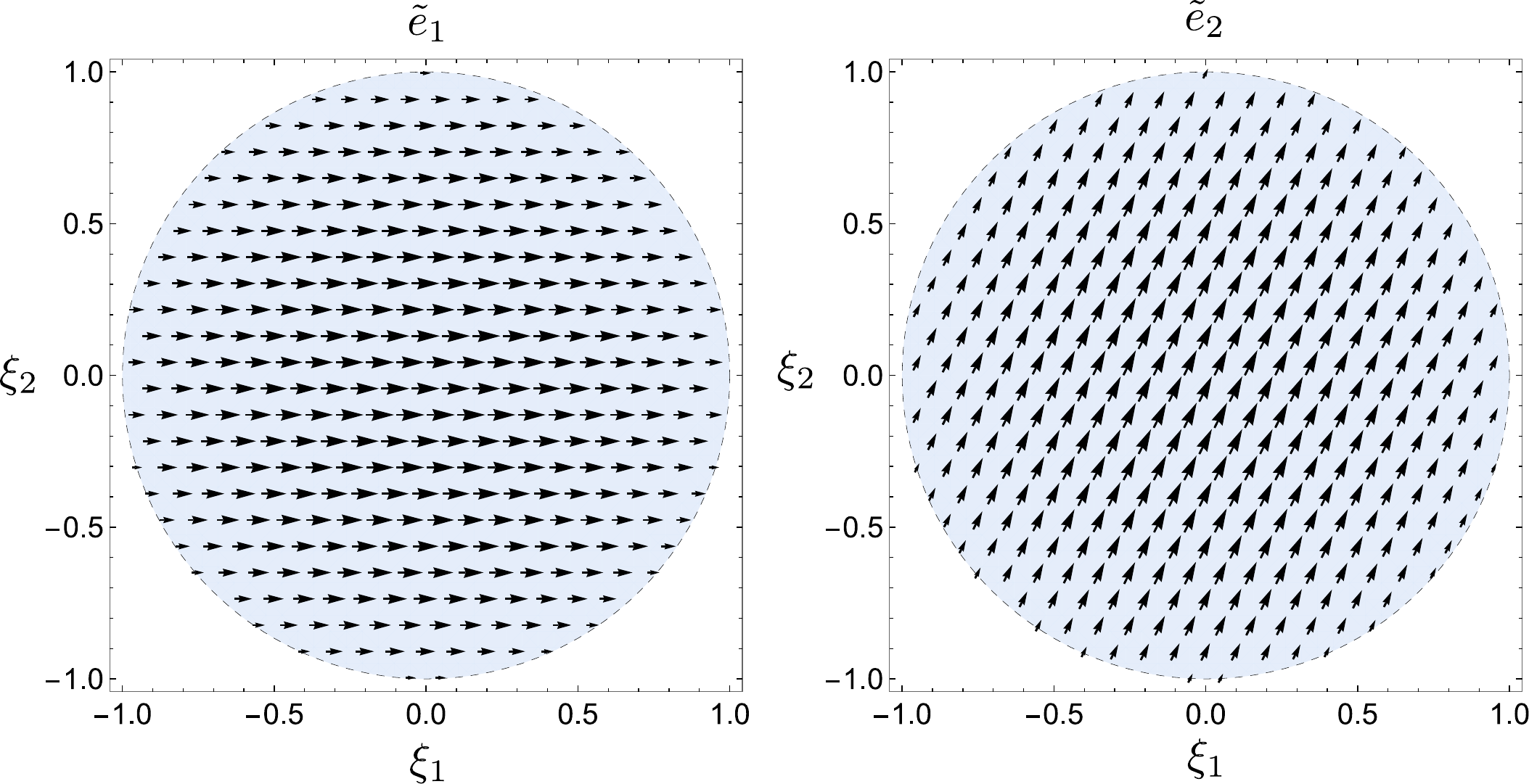}
			\caption{\normalfont An illustration of the tilted vector fields $\tilde{e}_{i}$ for Example \ref{exmp:sphere_exmp_deception} in the $\xi$-coordinates. The arrows' sizes represent the magnitude of the vector field.}
			\label{fig:S2_example_vectorfields_tilted}
			\vspace{-0.2cm}
		\end{figure}
		As observed, the vector field $\tilde{e}_2$ is tilted compared to Figure \ref{fig:sphere_example_vectorfields}.
		Repeating the same computations as in Example \ref{exmp:sphere_exmp_no_deception}, we find that the new frame $\{\tilde{e}_i\}$, with $\tilde{e}_1=e_1$ and $\tilde{e}_2$ as above, still spans the tangent space of the submanifold $\mathcal{U}$. In particular, we have that $dJ_1[e_1](x)= -x_1$, $dJ_2[\tilde{e}_{2}]= -x_2 - \gamma (x_1-x_3)$, so that the only equilibrium in $\mathcal{U}$ is given, as a function of the deception parameter $\gamma$, by
		\begin{align}\label{eq:s2_eqbm_gamm}
			X^\star(\gamma) = \left(0,\frac{\gamma}{\sqrt{1+\gamma^2}},\frac{1}{\sqrt{1+\gamma^2}}\right).
		\end{align}
		Next, we proceed to verify the properties of $X^\star(\gamma)$.
		By Theorem \ref{thm:deception_stability_IFT}, there exists an open neighborhood $\Omega \subset\mathbb{R}$ with $0\in \Omega$ such that, for all $\gamma\in\Omega$, $X^\star(\gamma)$ is locally exponentially stable. Nevertheless, because we have an explicit expression of the equilibrium point, we can obtain stronger results. To do so, we employ the frame $\{\tilde{e}_1,\tilde{e}_2\}$, which is orthonormal with respect to the tilted metric $\tilde{g}$ of Corollary \ref{cor:tilted_metric}. In particular, we compute that $d(dJ_1[\tilde{e}_1])[\tilde{e}_1] = - x_3$, $d(dJ_1[\tilde{e}_1])[\tilde{e}_2] = -\gamma  x_3$, $d(dJ_2[\tilde{e}_2])[\tilde{e}_1] = -\gamma  (x_1+x_3)$, $d(dJ_2[\tilde{e}_2])[\tilde{e}_2]= - \gamma (\gamma x_1+x_2)-(1+\gamma^2)x_3$. Hence, substituting from \eqref{eq:s2_eqbm_gamm}, we see that the matrix representation of $A|_{X^\star(\gamma)}$ with respect to $\{\tilde{e}_i\}$ is
		\begin{align*}
			A|_{X^\star(\gamma)} = \frac{1}{\sqrt{1+\gamma^2}} \begin{bmatrix} -1& -\gamma \\ -\gamma & -(1+2\gamma^2)\end{bmatrix}.
		\end{align*}
		Note that the diagonal entries of $A|_{X^\star(\gamma)}$ are strictly negative so that the point $X^\star(\gamma)$ is a strict local leaf-wise Nash equilibrium for all admissible $\gamma$. To verify that $A|_{X^\star(\gamma)} $ is an automorphism of $T_{X^\star(\gamma)}\mathcal{U}$ and Hurwitz, we compute its trace and determinant as $\mathrm{trace}(A|_{X^\star(\gamma)}) = -2\sqrt{1+\gamma^2}$ and $\mathrm{det}(A|_{X^\star(\gamma)})= 1$, which, since it is a $2\times 2$ matrix, imply that its eigenvalues are strictly negative, and, therefore, $X^\star(\gamma)$ is a non-degenerate isolated exponentially stable equilibrium for all $\gamma\in\mathbb{R}$. In conclusion, deception preserves the stability of the new equilibrium point \eqref{eq:s2_eqbm_gamm}, which is controlled by the deceiver via the parameter $\gamma$. We also note that differentiating \eqref{eq:s2_eqbm_gamm} at $\gamma=0$ gives $\frac{d}{d\gamma}X^\star(0) = (0,1,0) = e_2(x^\star)$, in agreement with Corollary \ref{cor:sensitivity}, since $dJ_2[e_1](x^\star)=1$ and $A|_{x^\star}=-I$. \hfill $\square$
	\end{exmp}
	\medskip
	\begin{exmp}[Example \ref{exmp:simplex_exmp_no_deception} cont'd]\label{exmp:simplex_exmp_deception}
		We consider the case in which the $1^{\text{st}}$-agent implements a deceptive strategy towards the $2^{\text{nd}}$-agent in the setting of Example \ref{exmp:simplex_exmp_no_deception}. In this case, the vector fields $\{e_i\}$ are replaced by $\tilde{e}_1(x) := e_1(x)$,  $\tilde{e}_2(x) := e_2(x) + \gamma\, e_1(x)$, where $\gamma \in \Real$ is the deception parameter. The new vector fields are depicted in Figure \ref{fig:P2_example_vectorfields_tilted} for $\gamma = \frac{1}{2}$. As can be observed from the figure, the vector field $\tilde{e}_2$ is tilted compared to Figure \ref{fig:P2_exmp_vectorfields}.
		\begin{figure}[t]
			\centering
			\includegraphics[width=0.8\linewidth]{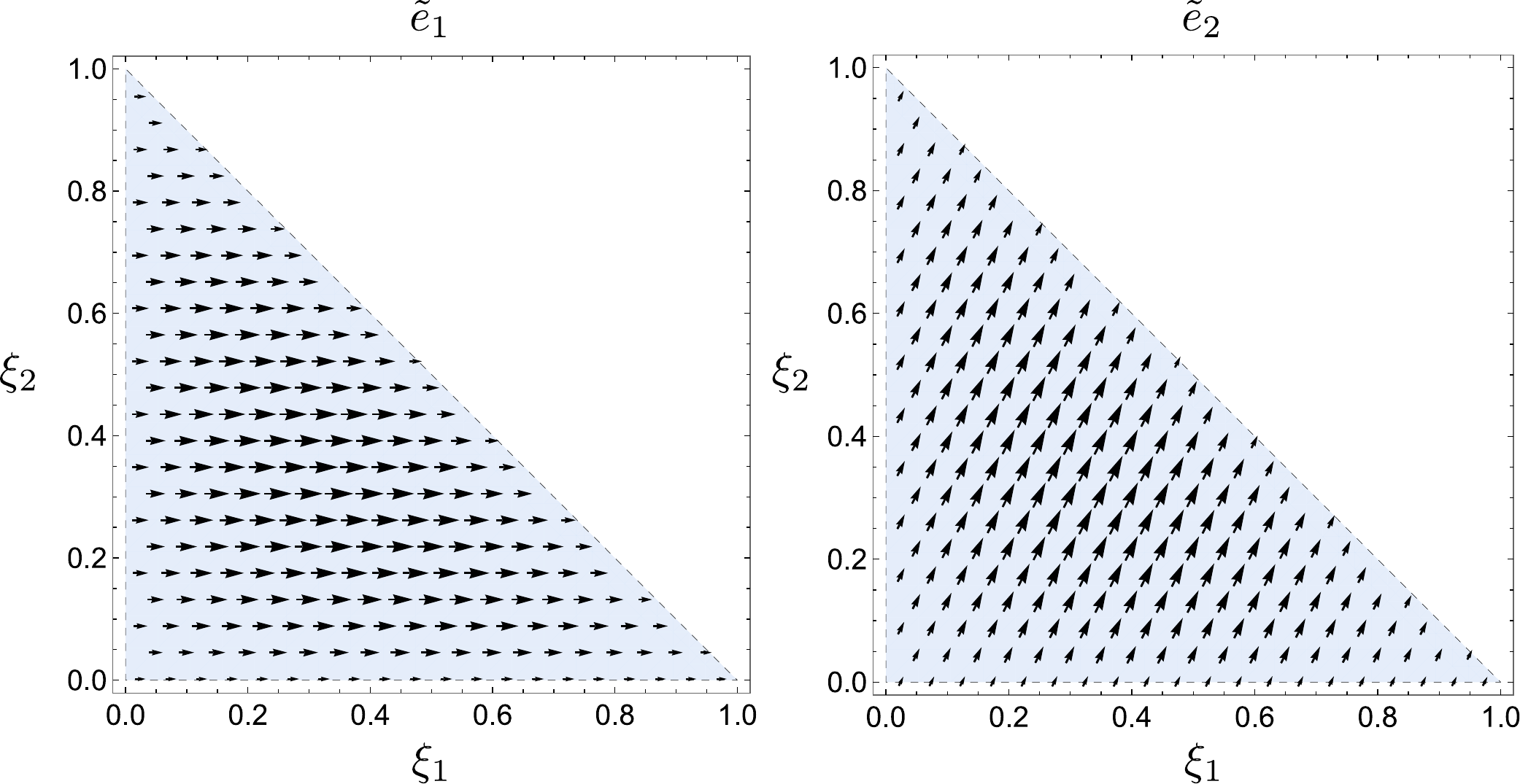}
			\caption{\normalfont An illustration of the tilted vector fields $\tilde{e}_{i}$ for Example \ref{exmp:simplex_exmp_deception} in the $\xi$-coordinates. The arrows' sizes represent the magnitude of the vector field.}
			\label{fig:P2_example_vectorfields_tilted}
		\end{figure}
		We compute that $dJ_2[e_1](x) = x_1 x_2 x_3\,(x_3 - x_1)$, and, consequently, $dJ_1[\tilde{e}_1](x) = x_1 x_2 x_3\,(-x_1 + x_3 + 1 - 3\alpha)$, $dJ_2[\tilde{e}_2](x) = x_1 x_2 x_3\big[(-x_2 + x_3 + 1 - 3\alpha) + \gamma(x_3 - x_1)\big]$, where we utilized \eqref{eq:dJ_i}. On $\Dom$ we have $x_1 x_2 x_3 > 0$, so the equilibrium conditions reduce to the two linear equations $-x_1 + x_3 + 1 - 3\alpha = 0$, $(-x_2 + x_3 + 1 - 3\alpha) + \gamma(x_3 - x_1) = 0$, together with the simplex constraint $x_1 + x_2 + x_3 = 1$. Introducing $\beta := 1 - 3\alpha$, the first equation gives $x_1 = x_3 + \beta$, so that $x_3 - x_1 = -\beta$. Substituting into the second and simplifying yields $x_2 = x_3 + \beta(1 - \gamma)$. Therefore, the simplex constraint determines $x_3 = \frac{(6\alpha - 1) + \gamma(1 - 3\alpha)}{3}$. From the above, we obtain that
		\begin{equation}
			X^\star(\gamma) = \frac{1}{3}\begin{pmatrix} (2 - 3\alpha) + \gamma(1 - 3\alpha) \\ (2 - 3\alpha) - 2\gamma(1 - 3\alpha) \\ (6\alpha - 1) + \gamma(1 - 3\alpha) \end{pmatrix}. \label{eq:eq}
		\end{equation}
		In particular, $X^\star(0)$ recovers the nominal equilibrium from Example \ref{exmp:simplex_exmp_no_deception}, and the $\gamma$-dependent shift $X^\star(\gamma) - X^\star(0) = \frac{\gamma (1 - 3\alpha)}{3}\,(1,\, -2,\, 1)$ is exactly linear in $\gamma$. For $X^\star(\gamma)$ to remain in the interior $\Dom$, we require $\gamma \in \Omega$, where $\Omega \subset \Real$ is the open interval
		\begin{equation}
			\Omega = \left\{\gamma \in \Real \;:\; X^\star_i(\gamma) > 0 \text{ for } i = 1, 2, 3 \right\}. \label{eq:Omega}
		\end{equation}
		The three positivity constraints in \eqref{eq:Omega} read $(2 - 3\alpha) + \gamma(1 - 3\alpha) > 0$,
		$(2 - 3\alpha) - 2\gamma(1 - 3\alpha) > 0$, $(6\alpha - 1) + \gamma(1 - 3\alpha)> 0$. For $\alpha \in (1/6, 1/3)$, these give that $\Omega = \left(\frac{1 - 6\alpha}{1 - 3\alpha},\; \frac{2 - 3\alpha}{2(1 - 3\alpha)}\right)$, which contains $0$ and is bounded on both sides. For $\alpha = 1/3$, the map $\gamma \mapsto X^\star(\gamma)$ is constant ($\beta = 0$), and deception has no effect on the equilibrium. For $\alpha \in (1/3, 2/3)$ (so $1 - 3\alpha < 0$), the inequalities flip direction and $\Omega = \left(\frac{2 - 3\alpha}{2(1 - 3\alpha)},\; \frac{1 - 6\alpha}{1 - 3\alpha}\right)$. In both non-degenerate regimes, $\Omega$ is an explicit, bounded, open interval containing $0$, giving a quantitative counterpart to the qualitative neighborhood guaranteed by Theorem \ref{thm:deception_stability_IFT}.
		
		\medskip
		
		To verify local optimality and stability of $X^\star(\gamma)$, we compute the intrinsic Jacobian $A|_{X^\star(\gamma)}$ with respect to the frame $\{\tilde{e}_1, \tilde{e}_2\}$, which is orthonormal with respect to the tilted metric $\tilde{g}$, via Theorem \ref{thm:jacobian_simplification}. Observe that each directional derivative above factors as $\chi \cdot \rho$, where $\chi := x_1 x_2 x_3$ and $\rho$ is a polynomial of degree $1$ in $x$. At $X^\star(\gamma)$, the linear factor $\rho$ vanishes by construction of the equilibrium, while $\chi|_{X^\star(\gamma)} > 0$. Hence, using $\tilde{e}_1 = \chi(\partial_1 - \partial_3)$ and $\tilde{e}_2 = \chi (\gamma\partial_1 + \partial_2 - (1+\gamma)\partial_3)$,
		\begin{align*}
			d(dJ_i[\tilde{e}_i])[\tilde{e}_j]\big|_{X^\star(\gamma)} &= \chi\cdot \tilde{e}_j[\rho_i]\big|_{X^\star(\gamma)},
		\end{align*}
		where $\rho_1 = -x_1 + x_3 + \beta$ and $\rho_2 = (-x_2 + x_3 + \beta) + \gamma(x_3 - x_1)$. Direct computation gives $d(dJ_1[\tilde{e}_1])[\tilde{e}_1]= -2\,\chi^2$, $d(dJ_1[\tilde{e}_1])[\tilde{e}_2]= -(1 + 2\gamma)\,\chi^2$,
		$d(dJ_2[\tilde{e}_2])[\tilde{e}_1]= -(1 + 2\gamma)\,\chi^2$,
		$d(dJ_2[\tilde{e}_2])[\tilde{e}_2]= -2(1 + \gamma + \gamma^2)\,\chi^2$. Hence, substituting from~\eqref{eq:eq}, the matrix representation of $A|_{X^\star(\gamma)}$ with respect to $\{\tilde{e}_i\}$ is
		\begin{equation}
			A|_{X^\star(\gamma)} = \chi^2\big|_{X^\star(\gamma)} \begin{bmatrix} -2 & -(1 + 2\gamma) \\ -(1 + 2\gamma) & -2(1 + \gamma + \gamma^2) \end{bmatrix}. \label{eq:A}
		\end{equation}
		In particular, at $\gamma = 0$, the scalar $\chi^2\big|_{x^\star} = \tfrac{1}{729}(2 - 3\alpha)^4 (6\alpha - 1)^2$ and the matrix~\eqref{eq:A} reduces to the expression in Example \ref{exmp:simplex_exmp_no_deception_2}. Note that the diagonal entries here are once more strictly negative so that the point $X^\star(\gamma)$ is a strict local leaf-wise Nash equilibrium for all admissible $\gamma$. Next, we verify that $A|_{X^\star(\gamma)}$ is Hurwitz by computing its trace and determinant:
		\begin{align*}
			\operatorname{tr}\!\big(A|_{X^\star(\gamma)}\big) &= -2\,\chi^2\big|_{X^\star(\gamma)}\cdot (2 + \gamma + \gamma^2), \\
			\det\!\big(A|_{X^\star(\gamma)}\big) &= \chi^4\big|_{X^\star(\gamma)} \cdot \Big[4(1 + \gamma + \gamma^2) - (1 + 2\gamma)^2\Big] \\
			&=3\,\chi^4\big|_{X^\star(\gamma)}.
		\end{align*}
		Since $2 + \gamma + \gamma^2 = (\gamma + 1/2)^2 + 7/4 > 0$ for all $\gamma \in \Real$, the trace is strictly negative for every $\gamma \in \Omega$. The determinant equals $3\,\chi^4|_{X^\star(\gamma)} > 0$ for every $\gamma \in \Omega$. Therefore, $A|_{X^{\star}(\gamma)}$ is Hurwitz and $X^\star(\gamma)$ is a non-degenerate isolated locally exponentially stable equilibrium for system \eqref{eq:liebra_sys_deception}, for every $\gamma \in \Omega$. Thus, deception preserves the stability of the system. 
		$\hfill\square$
	\end{exmp}
	%
	%
	\section{PROOFS}
	\label{sec:proofs}
	%
	\subsubsection{Proof of Theorem \ref{thm:pseduo_grad}}
	Under Assumption \ref{asmp:indpenendent_families}, and for each $x\in\mathcal U$, we have the (pointwise) direct-sum decomposition $T_x\mathcal U = \mathcal E_1(x)\oplus\cdots\oplus \mathcal E_N(x)$, where, thanks to Assumption \ref{asmp:indpenendent_families}, each $\mathcal E_i$ is a regular distribution on $\mathcal U$ generated by the family
	$\{f_{i,j}\}_{j=1}^{d_i}$. Consequently, using annihilators yields the canonical dual decomposition
	\begin{gather}
		T_x^\star\mathcal U = \mathcal E_1(x)^\circ \oplus \cdots \oplus \mathcal E_N(x)^\circ, \label{eq:dual_splitting}\\
		\mathcal E_i(x)^\circ := \mathrm{Ann}\Big(\bigoplus_{k\neq i}\mathcal E_k(x)\Big)\subset T_x^\star\mathcal U.
	\end{gather}
	Let $\pi_i(x):T_x^\star\mathcal U\to \mathcal E_i(x)^\circ$ denote the corresponding linear projections. For each $i\in\{1,\dots,N\}$, define a linear operator
	\begin{equation}\label{eq:Pi_def_new_notation}
		P_i(x)[\omega]
		\;:=\;
		\sum_{j=1}^{d_i} \big(\pi_i(x)\omega\big)\big(f_{i,j}(x)\big)\, f_{i,j}(x),
	\end{equation}
	for all $\omega\in T_x^\star\mathcal U$. By construction, $P_i(x)[\omega]\in \mathcal E_i(x)$ for all $\omega\in T_x^\star\mathcal U$. Moreover, since
	$dJ_i(x)\big(f_{i,j}(x)\big)=L_{f_{i,j}}(J_i)(x)$, system \eqref{eq:liebra_system} can be rewritten as
	\begin{equation}\label{eq:system_Pi_form_new_notation}
		\dot x = \sum_{i=1}^N\kappa_{i}  P_i(x)\big[dJ_i(x)\big].
	\end{equation}
	We now construct a metric $g$ on $\mathcal U$ via its co-metric.
	For each $x\in\mathcal U$, we define a
	symmetric bilinear form $g_x^{-1}$ on $T_x^\star\mathcal U$ by declaring it block-diagonal with respect to the
	splitting \eqref{eq:dual_splitting} and, for $\alpha,\beta\in T_x^\star\mathcal U$ with unique decompositions $\alpha=\sum_{i=1}^N \alpha_i$ and
	$\beta=\sum_{i=1}^N \beta_i$ such that $\alpha_i,\beta_i\in \mathcal E_i(x)^\circ$, defining
	\begin{equation}\label{eq:cometric_def_new_notation}
		\begin{aligned}
			g_x^{-1}(\alpha,\beta)
			&:=
			\sum_{i=1}^N g^{-1}_{i,x}(\alpha_i,\beta_i),\\
			g^{-1}_{i,x}(\alpha_i,\beta_i)
			&:=
			\sum_{j=1}^{d_i} \alpha_i\big(f_{i,j}(x)\big)\,\beta_i\big(f_{i,j}(x)\big).
		\end{aligned}
	\end{equation}
	%
	Since $\{f_{i,j}(x)\}_{j=1}^{d_i}$ spans $\mathcal E_i(x)$, the quadratic form
	$\alpha_i\mapsto g^{-1}_{i,x}(\alpha_i,\alpha_i)=\sum_{j=1}^{d_i}\alpha_i(f_{i,j}(x))^2$ is positive definite on
	$\mathcal E_i(x)^\circ$. Hence, $g_x^{-1}$ is positive definite on $T_x^\star\mathcal U$ for every $x\in\mathcal U$.
	Smoothness follows from the smoothness of the vector fields $f_{i,j}$ and of the projections $\pi_i(x)$ for regular
	distributions. Therefore, the tensor $g^{-1}$ defines a Riemannian co-metric on $\mathcal U$, and, consequently, a unique Riemannian
	metric $g$.
	
	Next, let $g_x^\sharp:T_x^\star\mathcal U\to T_x\mathcal U$ denote the tangent-cotangent isomorphism induced by $g$, i.e., $\alpha(v)=g_x\big(g_x^\sharp\alpha,v\big)$ for all $\alpha\in T_x^\star\mathcal U,\ v\in T_x\mathcal U$. By construction of \eqref{eq:cometric_def_new_notation}, we have the block formula
	\begin{equation}\label{eq:sharp_block_new_notation}
		g_x^\sharp\alpha = \sum_{i=1}^N P_i(x)[\alpha_i],
	\end{equation}
	where $\alpha=\sum_{i=1}^N \alpha_i,\ \alpha_i\in \mathcal E_i(x)^\circ$ is the unique decomposition of $\alpha$ with respect to the splitting \eqref{eq:dual_splitting}.
	For each $i$, the Riemannian gradient of the function $J_i$ with respect to the metric $g$ is defined by $\grad(J_i)(x):=g_x^\sharp\,(dJ_i(x))$. Utilizing the splitting \eqref{eq:dual_splitting} once more, we have the unique decomposition
	$dJ_i(x)=\sum_{k=1}^N (dJ_i(x))_k$ with $(dJ_i(x))_k\in \mathcal E_k(x)^\circ$, which, in combination with \eqref{eq:sharp_block_new_notation}, implies that $\grad(J_i)(x)=\sum_{k=1}^N P_k(x)\big[(dJ_i(x))_k\big]$. Moreover, since $g_x^{-1}$ is block-diagonal with respect to the splitting \eqref{eq:dual_splitting}, it follows that the decomposition \eqref{eq:tangent_splitting} is $g_x$-orthogonal, and, therefore, the $g_x$-orthogonal projection $P_{\mathcal E_i}(x):T_x\mathcal U\to \mathcal E_i(x)$ is precisely the extraction of the $i$-th component. Hence
	\[
	\grad_{\mathcal{E}_i}(J_i)(x):=P_{\mathcal E_i}[\grad(J_i)](x)
	=
	P_i(x)\big[(dJ_i(x))_i\big].
	\]
	Finally, since $f_{i,j}(x)\in \mathcal E_i(x)$ we have that $(dJ_i(x))_i(f_{i,j}(x))=dJ_i(x)(f_{i,j}(x))$, and, therefore, it follows that $\grad_{\mathcal{E}_i}(J_i)(x)=\sum_{j=1}^{d_i} L_{f_{i,j}}(J_i)(x)\,f_{i,j}(x)$. Combining this identity with \eqref{eq:system_Pi_form_new_notation} proves that system~\eqref{eq:liebra_system} is of the form \eqref{eq:projected_grad_form}, which concludes the proof.
	\subsubsection{Proof of Proposition \ref{prop:weighted-block-gershgorin-riem}}
	Let
	$
	X:=\mathrm{diag}(\chi_1 I_1,\dots,\chi_N I_N),
	$
	where \(I_i\) is the identity on \(\mathcal E_i(x^\star)\), and define the operator
	$
	\tilde A:=X^{-1}AX,
	$
	which is similar to $A$, i.e., has the same spectrum. Direct computation shows that its blocks with respect to the splitting \eqref{eq:tangent_splitting} are
	$
	\tilde A_{ii}=A_{ii},
	$
	and
	$
	\tilde A_{ij}=\frac{\chi_j}{\chi_i}A_{ij}.
	$
	Therefore,
	\[
	\tilde R_i
	:=
	\sum_{j\neq i}\|\tilde A_{ij}\|_{g}
	=
	\sum_{j\neq i}\|A_{ij}\|_{g}\frac{\chi_j}{\chi_i},
	\]
	with respect to \(\|\cdot\|_{g}\).
	By the block Gershgorin circle theorem \cite{feingold1962block}, each
	\(\lambda\in\sigma(\tilde A)\) satisfies, for some \(i\),
	\begin{align}\label{eq:gersh}
		\|(\tilde A_{ii}-\lambda I_i)^{-1}\|_{g}^{-1}
		\le \tilde R_i.
	\end{align}
	Since \(\tilde A_{ii}=A_{ii}\) and $
	\|M^{-1}\|_{g}^{-1}=\sigma_{\min}^g(M),$ we obtain that $\sigma_{\min}^g(A_{ii}-\lambda I_i)\le \tilde R_i.$
	To arrive at a contradiction, suppose that, for some $\lambda\in\sigma(\tilde{A})$, we have that \(\Re(\lambda)\ge 0\). Then, we would have that
	\[
	\sigma_{\min}^g(A_{ii}-\lambda I_i)
	\ge
	\inf_{\Re (s)\ge 0}\sigma_{\min}^g(A_{ii}-sI_i)
	>
	\tilde R_i
	\]
	from \eqref{eq:wbg-riem}, which contradicts \eqref{eq:gersh}. Hence, $\sigma(\tilde A)\subset\{\Re(z)<0\}$ and, because \(\sigma(A)=\sigma(\tilde A)\), \(A\) is Hurwitz.
	\subsubsection{Proof of Theorem \ref{thm:jacobian_simplification}}
	Recall from \eqref{eq:jacobian_block_form} that
	\begin{align}
		A_{ij}[\,\cdot\,] &= P_{\Ei_i}\left[\nabla_{P_{\Ei_j}[\,\cdot\,]}\,\grad(J_i)\right] + C_{ij}[\,\cdot\,], \label{eq:Aij}\\
		C_{ij}[\,\cdot\,] &= \sum_{k=1}^{N} P_{\Ei_i}\left[\big(\nabla_{P_{\Ei_j}[\,\cdot\,]} P_{\Ei_k}\big)[\grad(J_k)]\right]. \label{eq:Cij}
	\end{align}
	We evaluate $A_{ij}[v]$ for an arbitrary vector field $v$ and then extract its components in the orthonormal frame $\{f_{i,j}\}$.
	
	\medskip
	
	\noindent From~\eqref{eq:metric-proj} and the product rule, we compute that
	\[
	\nabla P_{\Ei_k} = \sum_{\ell=1}^{d_k}\Big( \nabla \vartheta^{k,\ell} \otimes f_{k,\ell} + \vartheta^{k,\ell} \otimes \nabla f_{k,\ell}\Big).
	\]
	Since $P_{\Ei_j}[v] = \sum_{q=1}^{d_j} \vartheta^{j,q}[v]\, f_{j,q}$, we obtain that
	\begin{equation}\label{eq:nabla-P}
		\begin{aligned}
			\nabla_{P_{\Ei_j}[v]} P_{\Ei_k} = \sum_{q=1}^{d_j} \vartheta^{j,q}[v]\, &\sum_{\ell=1}^{d_k}\Big( \vartheta^{k,\ell}\otimes \nabla_{f_{j,q}} f_{k,\ell} \\
			&+ (\nabla_{f_{j,q}} \vartheta^{k,\ell}) \otimes f_{k,\ell}\Big).
		\end{aligned}
	\end{equation}
	
	\medskip
	
	\noindent From the definition of the covariant derivative of a $1$-form,
	\[
	(\nabla_{f_{j,q}} \vartheta^{k,\ell})[\,\cdot\,] = d\big(\vartheta^{k,\ell}[\,\cdot\,]\big)[f_{j,q}]- \vartheta^{k,\ell}\big[\nabla_{f_{j,q}}[\,\cdot\,]\big].
	\]
	Applying this to $\grad(J_k)$ and utilizing \eqref{eq:dual} and \eqref{eq:grad},
	\begin{equation*}
		\begin{aligned}
			(\nabla_{f_{j,q}} \vartheta^{k,\ell})[\grad(J_k)] &= d\big(dJ_k[f_{k,\ell}]\big)[f_{j,q}] \\
			&- \vartheta^{k,\ell}\big[\nabla_{f_{j,q}}\grad(J_k)\big].
		\end{aligned}
	\end{equation*}
	Utilizing \eqref{eq:grad}, the product rule, and \eqref{eq:dual}:
	\begin{equation*}
		\begin{aligned}
			\vartheta^{k,\ell}\big[\nabla_{f_{j,q}}\grad(J_k)\big]
			&= d\big(dJ_k[f_{k,\ell}]\big)[f_{j,q}] \\
			&+ \sum_{m,r} dJ_k[f_{m,r}]\, \vartheta^{k,\ell}[\nabla_{f_{j,q}} f_{m,r}].
		\end{aligned}
	\end{equation*}
	Hence, we obtain that
	\begin{equation}
		(\nabla_{f_{j,q}} \vartheta^{k,\ell})[\grad(J_k)] = -\sum_{m,r} dJ_k[f_{m,r}]\, \vartheta^{k,\ell}[\nabla_{f_{j,q}} f_{m,r}]. \label{eq:dtheta-grad}
	\end{equation}
	
	\medskip
	
	\noindent Substituting~\eqref{eq:nabla-P} into~\eqref{eq:Cij} and applying $P_{\Ei_i}$, which extracts components along $\{f_{i,p}\}_{p=1}^{d_i}$ via \eqref{eq:metric-proj}, we obtain:
	\begin{align*}
		P_{\Ei_i}&\left[\big(\nabla_{P_{\Ei_j}[v]} P_{\Ei_k}\big)[\grad(J_k)]\right]\\
		&= \sum_{q,p} \left(\sum_{\ell} dJ_k[f_{k,\ell}]\, \vartheta^{i,p}[\nabla_{f_{j,q}} f_{k,\ell}]\right) \vartheta^{j,q}[v] f_{i,p} \\
		&\quad + \sum_{q,\ell} (\nabla_{f_{j,q}} \vartheta^{k,\ell})[\grad(J_k)] \vartheta^{j,q}[v] P_{\Ei_i}[f_{k,\ell}].
	\end{align*}
	Using $P_{\Ei_i}[f_{k,\ell}] = \delta_{ik}\, f_{i,\ell}$ and substituting~\eqref{eq:dtheta-grad}, the second sum contributes only when $k = i$:
	\[
	-\,\delta_{ik}\sum_{q,\ell}\left(\sum_{m,r} dJ_k[f_{m,r}]\, \vartheta^{k,\ell}[\nabla_{f_{j,q}} f_{m,r}]\right) \vartheta^{j,q}[v]f_{i,\ell}.
	\]
	Summing over $k$ and relabeling $\ell \mapsto p$, we obtain that
	\begin{equation}\label{eq:Cij-expanded}
		\begin{aligned}
			C_{ij}[v] = \sum_{q,p}&\bigg(\sum_{k,\ell}\big(dJ_k[f_{k,\ell}]\\ &- dJ_i[f_{k,\ell}]\big)\,\vartheta^{i,p}[\nabla_{f_{j,q}} f_{k,\ell}]\bigg)\vartheta^{j,q}[v] f_{i,p}.
		\end{aligned}
	\end{equation}
	From~\eqref{eq:grad} and the product rule,
	\begin{equation*}
		\begin{aligned}
			\nabla_{f_{j,q}}\grad(J_i) = \sum_{k,\ell}\Big( d\big(&dJ_i[f_{k,\ell}]\big)[f_{j,q}]\, f_{k,\ell} \\
			&+ dJ_i[f_{k,\ell}]\, \nabla_{f_{j,q}} f_{k,\ell}\Big).
		\end{aligned}
	\end{equation*}
	Extracting the $(i,p)$-component with $\vartheta^{i,p} \circ P_{\Ei_i}$ and using~\eqref{eq:dual},
	\begin{equation}\label{eq:first-term}
		\begin{aligned}
			P_{\Ei_i}[&\nabla_{P_{\Ei_j}[v]}\grad(J_i)] = \sum_{q,p}\bigg(d\big(dJ_i[f_{i,p}]\big)[f_{j,q}] \\
			&+ \sum_{k,\ell} dJ_i[f_{k,\ell}]\, \vartheta^{i,p}[\nabla_{f_{j,q}} f_{k,\ell}]\bigg) \vartheta^{j,q}[v]f_{i,p}.
		\end{aligned}
	\end{equation}
	\noindent Adding \eqref{eq:Cij-expanded} and \eqref{eq:first-term} and canceling identical terms, we get
	\begin{equation}\label{eq:Aij-full}
		\begin{aligned}
			A_{ij}[v] &= \sum_{q,p}\bigg( d\big(dJ_i[f_{i,p}]\big)[f_{j,q}] \\
			&+ \sum_{k,\ell} dJ_k[f_{k,\ell}]\,\vartheta^{i,p}[\nabla_{f_{j,q}} f_{k,\ell}]\bigg) \vartheta^{j,q}[v] f_{i,p}.
		\end{aligned}
	\end{equation}
	\medskip
	
	\noindent By condition \eqref{eq:eqbm_uncoupled} and~\eqref{eq:grad}, $\grad_{\Ei_k}(J_k)(x^\star) = 0$ implies $
	dJ_k[f_{k,\ell}]\big|_{x^\star} = 0$ for all $k, \ell$. Hence the inner sum in~\eqref{eq:Aij-full} vanishes at $x^\star$, leaving
	\[
	A_{ij}[v]\big|_{x^\star} = \sum_{q,p} d\big(dJ_i[f_{i,p}]\big)[f_{j,q}]\, \vartheta^{j,q}[v] f_{i,p}\big|_{x^\star}.
	\]
	Reading off the matrix entries with respect to the orthonormal frame gives the first equality in~\eqref{eq:block-entries}. Finally, using the identity $\hess(J_i)[X, Y] = d(dJ_i[Y])[X]- dJ_i[\nabla_X Y]$, applied with $X = f_{j,q}$ and $Y = f_{i,p}$,
	\[
	d\big(dJ_i[f_{i,p}]\big)[f_{j,q}] = \hess(J_i)[f_{i,p}, f_{j,q}] + dJ_i\big[\nabla_{f_{j,q}} f_{i,p}\big],
	\]
	establishing the second equality in~\eqref{eq:block-entries}.
	\subsubsection{Proof of Proposition \ref{prop:involutive_dists}}
	By the first equality in \eqref{eq:block-entries},
	$$(A_{ii}\vert_{x^\star})_{p,q} = d(dJ_i[f_{i,p}])[f_{i,q}]\vert_{x^\star}.$$
	The definition of the Hessian implies that
	\begin{equation}
		\label{eq:hess_decomp}
		d(dJ_i[f_{i,p}])[f_{i,q}]
		= \hess(J_i)[f_{i,q}, f_{i,p}]
		+ dJ_i[\nabla_{f_{i,q}} f_{i,p}].
	\end{equation}
	Since $\mathcal{U}_i$ is a weakly embedded submanifold,
	$$\nabla_{f_{i,q}} f_{i,p} = \nabla^i_{f_{i,q}} f_{i,p} + \Pi^i(f_{i,q}, f_{i,p}),$$
	where $\nabla^i_{f_{i,q}} f_{i,p} \in \mathcal{E}_i$, and $\Pi^i$ is the second fundamental form. The equilibrium condition \eqref{eq:eqbm_uncoupled} gives $dJ_i[f_{i,r}]\vert_{x^\star} = 0$ for all $r$, so
	$dJ_i[\nabla^i_{f_{i,q}} f_{i,p}]\vert_{x^\star} = 0$,
	and hence $dJ_i[\nabla_{f_{i,q}} f_{i,p}]\vert_{x^\star} = dJ_i[\Pi^i(f_{i,q}, f_{i,p})]\vert_{x^\star}$. From \cite[Proposition 1.2]{dajczer2019submanifold}, which is obtained by applying the Gauss equation of Riemannian geometry to the Hessian, we obtain that
	\begin{equation}
		\begin{aligned}
			\hess(J_i^{\mathcal{U}_i})[f_{i,q}, f_{i,p}] = \hess^i(&J_i)[f_{i,q}, f_{i,p}]\\ &- dJ_i[\Pi^i(f_{i,q}, f_{i,p})].
		\end{aligned}
	\end{equation}
	Substituting into~\eqref{eq:hess_decomp}, the second fundamental form contributions cancel, and we obtain $d(dJ_i[f_{i,p}])[f_{i,q}]\big\vert_{x^\star}= \hess^i(J_i^{\mathcal{U}_i})[f_{i,p}, f_{i,q}]\big\vert_{x^\star}.$
	\subsubsection{Proof of Proposition \ref{prop:commuting_vectorfields}}
	Since $\{f_{i,j}\}$ is a global orthonormal frame on $\mathcal{U}$ with respect to $g$ (see the proof of Theorem \ref{thm:pseduo_grad}), we have $g(f_{i,p}, f_{k,r}) = \delta_{ik}\delta_{pr}$, which is constant on $\mathcal{U}$. Therefore, all Lie-derivative
	terms in the Koszul formula \cite[Corollary 5.11]{lee2018introduction} vanish identically, and for any $k \in \{1,\ldots,N\}$ and $r \in \{1,\ldots,d_k\}$, we have that
	\begin{equation}
		\label{eq:koszul_reduced}
		\begin{aligned}
			2g(\nabla_{f_{j,q}} &f_{i,p},\, f_{k,r})
			= -g\!\left(f_{j,q},\, [f_{i,p}, f_{k,r}]\right)\\
			&+ g\!\left(f_{i,p},\, [f_{k,r}, f_{j,q}]\right)+ g\!\left(f_{k,r},\, [f_{j,q}, f_{i,p}]\right).
		\end{aligned}
	\end{equation}
	Fixing $i \neq j$, the last term vanishes for all $k$, since, by Assumption \ref{asmp:commuting_vectorfields},
	$[f_{j,q}, f_{i,p}] = 0.$
	Next, we consider each possible value for $k$:
	\begin{itemize}
		\item For $k \notin \{i,j\}$, Assumption \ref{asmp:commuting_vectorfields} gives that
		$[f_{i,p}, f_{k,r}] = 0, \quad [f_{k,r}, f_{j,q}] = 0,$
		so the RHS in \eqref{eq:koszul_reduced} vanishes.
		\item For $k = i$, $[f_{i,p}, f_{k,r}] \in \mathcal{E}_i$ by Assumption \ref{asmp:involutive_distributions},
		so
		$g(f_{j,q}, [f_{i,p}, f_{i,r}]) = 0,$
		since $\mathcal{E}_j \perp_g \mathcal{E}_i$ (by construction), and, by
		Assumption \ref{asmp:commuting_vectorfields},
		$[f_{i,r}, f_{j,q}] = 0$.
		\item For $k = j$,
		Assumption \ref{asmp:commuting_vectorfields} gives that
		$[f_{i,p}, f_{k,r}] = 0$,
		whereas, by
		Assumption \ref{asmp:involutive_distributions}, we have that
		$[f_{j,r}, f_{j,q}] \in \mathcal{E}_j$, and, therefore $g(f_{i,p}, [f_{j,r}, f_{j,q}]) = 0$,
		since $\mathcal{E}_i \perp_g \mathcal{E}_j$ (by construction).
	\end{itemize}
	From the above, it follows that $g(\nabla_{f_{j,q}} f_{i,p}, f_{k,r}) = 0,$
	for all $k,r$, and, since $\{f_{k,r}\}$ is a global frame, \eqref{eq:vanishing_connection} follows.
	\subsubsection{Proof of Proposition \ref{propLNE}} Fix an arbitrary $i\in\{1,\dots,N\}$. By Assumption \ref{asmp:involutive_distributions}, the distribution $\mathcal{E}_i$ is involutive.
	Therefore, the Frobenius Theorem guarantees that
	$\mathcal{U}_i^\star=\mathcal{U}_i(x^\star)$ is the maximal connected integral manifold of $\mathcal{E}_i$ through $x^\star$ and satisfies \eqref{leaf}. We equip $\mathcal{U}_i^\star$ with the Riemannian metric
	$g^i$ induced by the ambient metric $g$.
	
	We first show that \eqref{eq:eqbm_uncoupled} is
	equivalent to the first-order stationarity of
	$J_i|_{\mathcal{U}_i^\star}$ at $x^\star$. Indeed, for every
	$v_i\in T_{x^\star}\mathcal{U}_i^\star=\mathcal{E}_i(x^\star)$, we have
	\[
	\begin{aligned}
		d\big(&J_i|_{\mathcal{U}_i^\star}\big)(x^\star)[v_i]
		=
		dJ_i(x^\star)[v_i] =
		g\bigl(\operatorname{grad}(J_i)(x^\star),v_i\bigr)\\
		&=
		g\bigl(
		P_{\mathcal{E}_i}\operatorname{grad}(J_i)(x^\star),
		v_i
		\bigr)=
		g^i\bigl(
		\operatorname{grad}_{\mathcal{E}_i}(J_i)(x^\star),
		v_i
		\bigr).
	\end{aligned}
	\]
	Thus, \eqref{eq:eqbm_uncoupled} implies $d\big(J_i|_{\mathcal{U}_i^\star}\big)(x^\star)=0$. Next, we apply the second-order sufficient condition on the
	Riemannian manifold $\mathcal{U}_i^\star$. Choose a smooth coordinate chart $\varphi_i:\mathcal{O}_i'\longrightarrow
	\varphi_i(\mathcal{O}_i')\subset\mathbb{R}^{d_i}$ on $\mathcal{U}_i^\star$ such that
	$x^\star\in\mathcal{O}_i'$ and
	$\varphi_i(x^\star)=0$. Define the coordinate representation $h_i
	:=
	\big(J_i|_{\mathcal{U}_i^\star}\big)\circ\varphi_i^{-1}$. Since $d\big(J_i|_{\mathcal{U}_i^\star}\big)(x^\star)=0$, it follows that $Dh_i(0)=0$. At a critical point, the coordinate representation of the
	Riemannian Hessian coincides with the ordinary second
	derivative because the connection-dependent terms are
	multiplied by the first derivative, which vanishes at the
	critical point. Therefore, condition
	\eqref{eq:leafwise_second_order} implies that the second Fr\'echet derivative of $h_i$ satisfies $z^\top D^2 h_i(0)z<0$, for all $z\in\mathbb{R}^{d_i}\setminus\{0\}$. Hence, $D^2 h_i(0)$ is negative definite. It follows
	that there exists a constant $c_i>0$ such that $z^\top D^2 h_i(0)z
	\leq -2c_i\|z\|^2$, for all $z\in\mathbb{R}^{d_i}$. By continuity of $D^2 h_i$, there exists $r_i>0$
	such that $B_{r_i}(0)$ is contained in
	$\varphi_i(\mathcal{O}_i')$ and $z^\top D^2 h_i(\xi)z
	\leq -c_i\|z\|^2$, for every $\xi\in B_{r_i}(0)$ and every
	$z\in\mathbb{R}^{d_i}$. For any $z\in B_{r_i}(0)$, Taylor's theorem with integral
	remainder and $Dh_i(0)=0$ give
	\[
	\begin{aligned}
		h_i(z)-h_i(0)
		&\leq
		-c_i\|z\|^2\int_0^1(1-\tau)\,d\tau=
		-\frac{c_i}{2}\|z\|^2.
	\end{aligned}
	\]
	Consequently, $h_i(z)<h_i(0)$ for all $z\in B_{r_i}(0)\setminus\{0\}$. Defining $\mathcal{O}_i
	:=
	\varphi_i^{-1}\bigl(B_{r_i}(0)\bigr)$, we obtain a neighborhood of $x^\star$ in
	$\mathcal{U}_i(x^\star)$ such that $J_i(y)<J_i(x^\star)$, for all $y\in\mathcal{O}_i\setminus\{x^\star\}$. Because $i$ was arbitrary, this property holds for every
	player. Therefore, $x^\star$ is a strict local leafwise Nash
	equilibrium. \hfill $\blacksquare$
	\subsubsection{Proof of Proposition \ref{standardNE}} We show that, under Assumption~\ref{asmp:commuting_vectorfields}, a strict local leafwise Nash equilibrium is precisely a strict local Nash equilibrium in the local product coordinates, and then invoke Proposition~\ref{propLNE}. By Assumption~\ref{asmp:indpenendent_families}, the tangent bundle admits the direct-sum decomposition $T\mathcal{U}=\mathcal{E}_1\oplus\cdots\oplus \mathcal{E}_N$. Assumption~\ref{asmp:involutive_distributions} guarantees that every distribution $\mathcal{E}_i$ is
	involutive, while Assumption~\ref{asmp:commuting_vectorfields} guarantees that vector fields
	belonging to distinct distributions commute. Therefore, after
	possibly restricting $\mathcal{U}$ to a sufficiently small neighborhood
	of $x^\star$, which we do not relabel, the Frobenius Theorem yields the
	local diffeomorphism $\phi:\mathcal{U}
	\longrightarrow
	\mathcal{U}_1\times\cdots\times\mathcal{U}_N$ introduced above, which satisfies $\phi_*\mathcal{E}_i=T\mathcal{U}_i$ for all $i\in\{1,\dots,N\}$.
	
	Let $z^\star:=\phi(x^\star)
	=(z_1^\star,\dots,z_N^\star)$ and define $\widehat{J}_i:=J_i\circ\phi^{-1}$. For each player $i$, consider the unilateral payoff function
	\[
	\widehat{J}_i^{\,z_{-i}^\star}:
	\mathcal{U}_i\longrightarrow\mathbb{R},
	\qquad
	\widehat{J}_i^{\,z_{-i}^\star}(z_i)
	:=
	\widehat{J}_i(z_i,z_{-i}^\star).
	\]
	Fix an arbitrary $i\in\{1,\dots,N\}$ and identify the product slice $\big\{(z_i,z_{-i}^\star):z_i\in\mathcal{U}_i\big\}$ with $\mathcal{U}_i$ through its coordinate expression $z_i\mapsto(z_i,z_{-i}^\star)$. Since $\phi_*\mathcal{E}_i=T\mathcal{U}_i$, the map $\iota_i(z_i):=\phi^{-1}(z_i,z_{-i}^\star)$ is a diffeomorphism from $\mathcal{U}_i$ onto an open neighborhood of $x^\star$ in the leaf $\mathcal{U}_i^\star=\mathcal{U}_i(x^\star)$, with $\iota_i(z_i^\star)=x^\star$ and $\widehat{J}_i^{\,z_{-i}^\star}=J_i|_{\mathcal{U}_i^\star}\circ\iota_i$. That is, the slice is the coordinate expression of the leaf, and the unilateral payoff of player $i$ is the coordinate expression of $J_i|_{\mathcal{U}_i^\star}$. Consequently, $z_i^\star$ is a strict local maximizer of $\widehat{J}_i^{\,z_{-i}^\star}$ if and only if $x^\star$ is a strict local maximizer of $J_i|_{\mathcal{U}_i^\star}$. Since $i$ was arbitrary, $x^\star$ is a strict local Nash equilibrium in the local product coordinates if and only if it is a strict local leafwise Nash equilibrium.
	
	It therefore suffices to verify the hypotheses of Proposition~\ref{propLNE}. Assumption~\ref{asmp:involutive_distributions} and condition~\eqref{eq:eqbm_uncoupled} hold by hypothesis. Under Assumptions~\ref{asmp:indpenendent_families}--\ref{asmp:involutive_distributions}, Proposition~\ref{prop:involutive_dists} gives the expression \eqref{hessianexpression} for $A_{ii}|_{x^\star}$, which is the matrix representation, with respect to the orthonormal frame $\{f_{i,p}(x^\star)\}_{p=1}^{d_i}$, of the Riemannian Hessian $\hess^{i}\big(J_i|_{\mathcal{U}_i^\star}\big)(x^\star)$, computed with respect to the metric $g^i$. Hence, condition~\eqref{eq:product_nash_second_order} implies condition~\eqref{eq:leafwise_second_order}, and Proposition~\ref{propLNE} shows that $x^\star$ is a strict local leafwise Nash equilibrium. By the equivalence established above, $x^\star$ is a strict local Nash equilibrium in
	the local product coordinates induced by the distributions
	$\{\mathcal{E}_i\}_{i=1}^N$. \hfill $\blacksquare$
	\subsubsection{Proof of Proposition \ref{prop:tilted_directions_independent}}
	
	First, we fix an order of the global frame as $(f_{1,1},\ldots,f_{1,d_1},f_{2,1},\ldots,f_{2,d_2})$,
	and let $n := d_1 + d_2$. From \eqref{eq:tilted_directions}, the linear map sending $f_{i,p} \mapsto \tilde{f}_{i,p}$ has the matrix representation $I + \Gamma^\top$ with respect to this ordered frame, where
	\begin{equation}
		\Gamma = \begin{bmatrix} 0_{d_1 \times d_1} & \gamma \\ 0_{d_2 \times d_1} & 0_{d_2 \times d_2} \end{bmatrix}, \label{eq:Gamma}
	\end{equation}
	with $\gamma \in \Real^{d_1 \times d_2} $ being the matrix of deception parameters $\gamma_{j}^{k}$.
	To prove the proposition, it suffices to show that $I + \Gamma^\top$, or, equivalently, $I+\Gamma$, is invertible, since, in that case, $\{\tilde{f}_{i,j}(x)\}$ is a basis of $T_x\Dom$, the rank-$d_i$ subspaces $\tilde{\Ei}_i(x)$ are well-defined, and the direct-sum decomposition $T_x\Dom = \tilde{\Ei}_1(x) \oplus \tilde{\Ei}_2(x)$ follows from the basis property. Since the bottom-left block of $\Gamma$ is zero, the matrix $I + \Gamma$ is block lower-triangular, and consequently $\det(I + \Gamma) = \det(I_{d_1}) \cdot \det(I_{d_2})= 1$, so $I + \Gamma$ is invertible at every $x \in \Dom$. Since $\{f_{i,j}(x)\}$ is a basis of $T_x\Dom$ at every $x \in \Dom$ by Assumption \ref{asmp:indpenendent_families}, and $I + \Gamma$ is invertible, the tilted family $\{\tilde{f}_{i,j}(x)\}$ is also a basis of $T_x\Dom$. In particular, $\{\tilde{f}_{i,j}(x)\}_{j=1}^{d_i}$ are linearly independent and span the rank-$d_i$ subspace $\tilde{\Ei}_i(x)$, and the decomposition $T_x\Dom = \tilde{\Ei}_1(x) \oplus \tilde{\Ei}_2(x)$ follows from dimension counting.
	\subsubsection{Proof of Theorem \ref{thm:deception_stability_IFT}}
	The proof proceeds in three steps. First, we recast the equilibrium condition \eqref{eq:eqbm_tilted} as a smooth map $F : \Dom \times \Real^{d_1 \times d_2} \to \Real^n$ whose zero set parametrizes the equilibria of \eqref{eq:liebra_sys_deception} as a function of $\gamma$. Second, we apply the implicit function theorem at $(x^\star, 0)$ using the non-degeneracy of the nominal Jacobian $A|_{x^\star}$ to obtain the smooth map $X^\star$. Third, we use continuity of the spectrum of the deceptive Jacobian $\tilde{A}|_{X^\star(\gamma)}$ to transfer local exponential stability from $\gamma = 0$ to a neighborhood $\Omega$. Let $\{\vartheta^{i,j}\}$ denote the dual co-frame of $\{f_{i,j}\}$ on $\Dom$, as in \eqref{eq:dual}. The equilibrium condition \eqref{eq:eqbm_tilted} for system \eqref{eq:liebra_sys_deception} reads $\grad_{\tilde{\Ei}_1}(J_1)(x) = 0$, $\grad_{\tilde{\Ei}_2}(J_2)(x) = 0$. Since $\tilde{f}_{1,j} = f_{1,j}$, the first condition is independent of $\gamma$ and amounts to $dJ_1[f_{1,j}](x) = 0$, for all $j = 1, \ldots, d_1$. Since $\tilde{f}_{2,j} = f_{2,j} + \sum_{k=1}^{d_1} \gamma_j^{k}\, f_{1,k}$, the second condition is
	\[
	dJ_2[\tilde{f}_{2,j}](x) = dJ_2[f_{2,j}](x) + \sum_{k=1}^{d_1}\gamma_j^{k}\, dJ_2[f_{1,k}](x) = 0,
	\]
	for $j = 1, \ldots, d_2$. We now introduce the smooth map $F : \Dom \times \Real^{d_1 \times d_2} \to \Real^n$ (with $n = d_1 + d_2$) by its $n$ scalar components
	\begin{equation}\label{eq:F}
		\begin{aligned}
			F_{1,j}(x, \gamma) &= dJ_1[f_{1,j}](x), \\
			F_{2,j}(x, \gamma) &= dJ_2[f_{2,j}](x) + \sum_{k=1}^{d_1} \gamma_j^{k}\, dJ_2[f_{1,k}](x).
		\end{aligned}
	\end{equation}
	Then, $x \in \Dom$ is an equilibrium of \eqref{eq:liebra_sys_deception} with deception parameter $\gamma$ if and only if $F(x, \gamma) = 0$. By construction, $F(x^\star, 0) = 0$ (the nominal equilibrium conditions \eqref{eq:eqbm_uncoupled} at $x^\star$).
	
	\medskip
	
	The partial derivative $\partial_x F(x^\star, 0) : T_{x^\star}\Dom \to \Real^n$ is the linear map with scalar components, for $v \in T_{x^\star}\Dom$,
	\begin{align*}
		\partial_x F_{1,j}(x^\star, 0)[v] &= d\big(dJ_1[f_{1,j}]\big)[v]\big|_{x^\star}, \\
		\partial_x F_{2,j}(x^\star, 0)[v] &= d\big(dJ_2[f_{2,j}]\big)[v]\big|_{x^\star}.
	\end{align*}
	Expanding $v = \sum_{i',p'} \vartheta^{i',p'}[v]\, f_{i',p'}(x^\star)$ and using the orthonormality of the frame, we obtain that
	\[
	\partial_x F_{i,j}(x^\star, 0)[v] = \sum_{i',p'} d\big(dJ_i[f_{i,j}]\big)[f_{i',p'}]\big|_{x^\star}\, \vartheta^{i',p'}[v].
	\]
	By the first equality in \eqref{eq:block-entries} of Theorem \ref{thm:jacobian_simplification}, the coefficients $d(dJ_i[f_{i,j}])[f_{i',p'}]|_{x^\star}$ are precisely the entries of the matrix representation of the nominal Jacobian $A|_{x^\star}$ with respect to the orthonormal frame $\{f_{i,j}\}$. Therefore, under the identification $\Real^n \simeq T_{x^\star}\Dom$ provided by the frame $\{f_{i,j}(x^\star)\}$, the linear map $\partial_x F(x^\star, 0)$ coincides with the matrix representation of $A|_{x^\star}$:
	\begin{equation}
		\partial_x F(x^\star, 0) \;\simeq\; A|_{x^\star}. \label{eq:F-to-A}
	\end{equation}
	By the hypothesis of local exponential stability of $x^\star$ and the ``only if" direction of Proposition \ref{prop:linearization_stability}, $A|_{x^\star}$ is Hurwitz and, in particular, invertible. Hence $\partial_x F(x^\star, 0)$ is a linear isomorphism from $T_{x^\star}\Dom$ to $\Real^n$.

	Next, note that the map $F$ defined in~\eqref{eq:F} is smooth in both arguments (since $J_1, J_2, f_{i,j}$ are all smooth), $F(x^\star, 0) = 0$, and $\partial_x F(x^\star, 0)$ is a linear isomorphism. By the implicit function theorem, there exists an open neighborhood $\Omega \subset \Real^{d_1 \times d_2}$ of $\gamma = 0$ and a smooth map $X^\star : \Omega \to \Dom$ such that $X^\star(0) = x^\star$ and $F(X^\star(\gamma), \gamma) = 0$ for all $\gamma \in \Omega$. By construction, $X^\star(\gamma)$ is an equilibrium of \eqref{eq:liebra_sys_deception} with deception parameter $\gamma$ for every $\gamma \in \Omega$, and uniqueness in a neighborhood of $x^\star$ is part of the statement of the implicit function theorem.
	
	Finally, let $\tilde{A}(\gamma) := \tilde{A}|_{X^\star(\gamma)}$ denote the intrinsic Jacobian of the deceptive dynamics \eqref{eq:liebra_sys_deception}, evaluated at the deceptive equilibrium $X^\star(\gamma)$. Since both $X^\star$ and the frame $\{\tilde{f}_{i,j}\} = \{\tilde{f}_{i,j}(\gamma)\}$ depend smoothly on $\gamma$, and since the projected gradients $\grad_{\tilde{\Ei}_i}(J_i)$ depend smoothly on $\gamma$ through $\tilde{f}_{i,j}$, the Jacobian $\tilde{A}(\gamma)$ is a smooth function of $\gamma$ on $\Omega$. At $\gamma = 0$, the deceptive frame $\{\tilde{f}_{i,j}\}$ coincides with the nominal frame $\{f_{i,j}\}$, the deceptive metric $\tilde{g}$ coincides with $g$, the deceptive bundles $\tilde{\Ei}_i$ coincide with $\Ei_i$, and $X^\star(0) = x^\star$. Consequently, $\tilde{A}(0) = A|_{x^\star},$
	which is Hurwitz by hypothesis. Since the spectrum of a matrix depends continuously on its entries, and smoothness of $\tilde{A}(\gamma)$ in $\gamma$ implies continuity of its entries, there exists an open neighborhood $\Omega' \subset \Omega$ of $\gamma = 0$ such that $\tilde{A}(\gamma)$ remains Hurwitz for all $\gamma \in \Omega'$. Shrinking $\Omega$ to $\Omega'$ if necessary, the ``if" direction of the second statement of Proposition \ref{prop:linearization_stability} applied to \eqref{eq:liebra_sys_deception} gives that $X^\star(\gamma)$ is a locally exponentially stable equilibrium of \eqref{eq:liebra_sys_deception} for all $\gamma \in \Omega$.
	%

	\section{Conclusion}
	We studied model-free multi-agent decision making on manifolds with payoff-only measurements and control-affine actuation along state-dependent vector fields. We showed that, under a bundle-splitting condition, the interaction dynamics admits an intrinsic representation as a non-holonomic gradient play, i.e., a sum of projected Riemannian gradients, which characterizes the equilibria and yields a local stability test via a connection-based linearization. Finally, in an asymmetric-information model, we showed that a deceiver who embeds the victims' exploration signals into its strategy tilts their effective actuation directions, which generically shifts the equilibria to locations of the deceiver's choosing.
	
	\bibliographystyle{ieeetr}
	\bibliography{references}
	
	\vspace{-0.7cm}
	\begin{IEEEbiography}
		[{\includegraphics[width=1in,height=1.25in,clip,keepaspectratio]{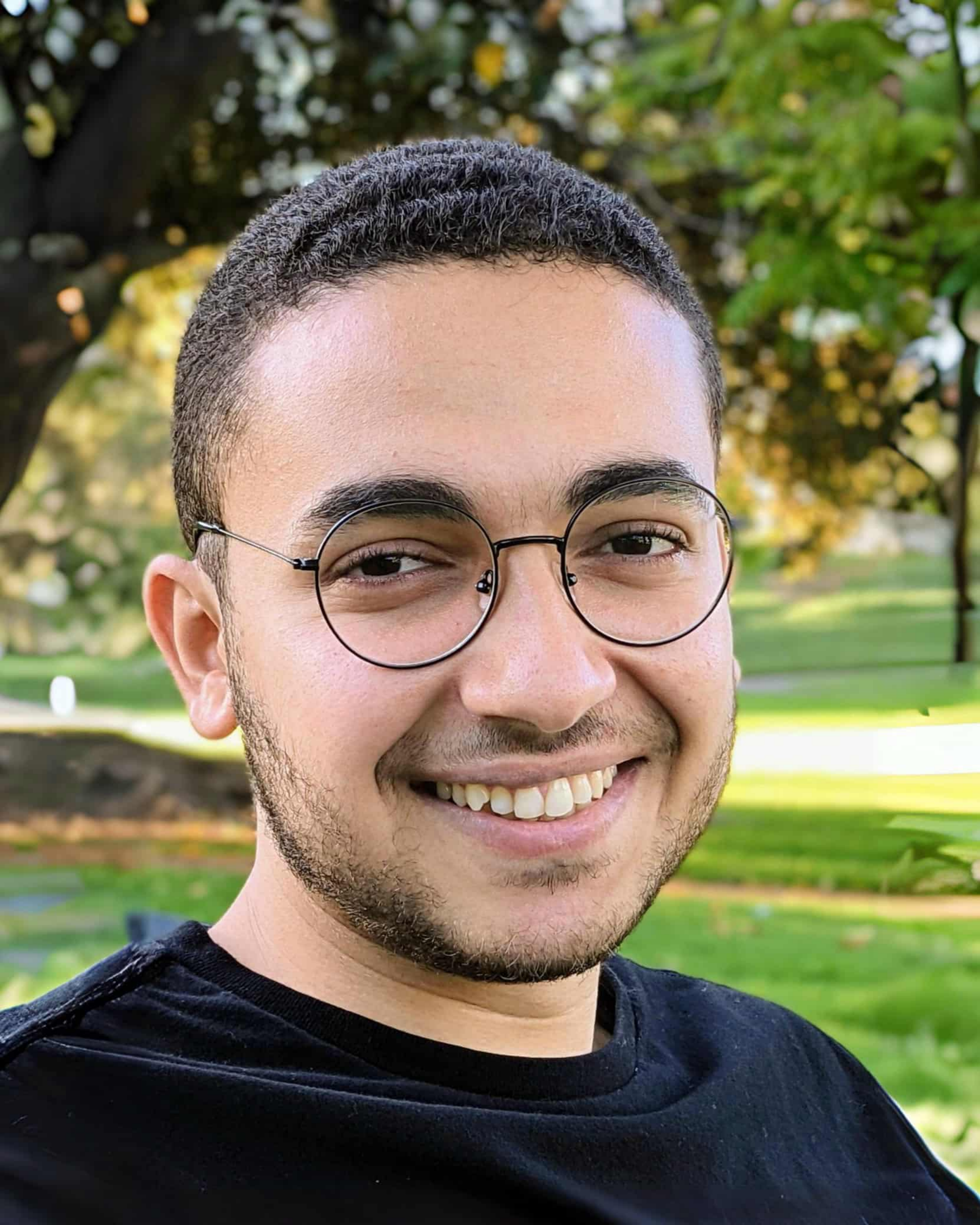}}]{Mahmoud Abdelgalil} is Assistant Professor in the Mechanical and Aerospace Engineering department at the University at Buffalo, State University of New York, Buffalo, NY. He received his BSc degree in Aerospace Engineering from Zewail University, Egypt, in 2018, and his MSc and PhD in Mechanical and Aerospace Engineering from the University of California, Irvine, in 2020 and 2023, respectively, followed by a Postdoctoral appointment in the Electrical and Computer Engineering department at the University of California, San Diego. He is a recipient of the Holmes fellowship and the Henry Samueli fellowship at the University of California, Irvine, and the awardee of the 2023 ACC Best Student Paper.
	\end{IEEEbiography}
	
	\vspace{-0.7cm}
	\begin{IEEEbiography}[{\includegraphics[width=1.09in,height=1.25in,clip,keepaspectratio]{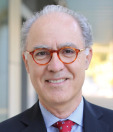}}]{Miroslav Krstic} Miroslav Krstic is Distinguished Professor of Mechanical and Aerospace Engineering, holds the Alspach endowed chair, and is the founding director of the Center for Control Systems and Dynamics at UC San Diego. He also serves as Senior Associate Vice Chancellor for Research at UCSD. As a graduate student, Krstic won the UC Santa Barbara best dissertation award and student best paper awards at CDC and ACC. Krstic has been elected Fellow of IEEE, IFAC, ASME, SIAM, AAAS, IET (UK), and AIAA (Assoc. Fellow) - and as a foreign member of the Serbian Academy of Sciences and Arts and of the Academy of Engineering of Serbia. He has received several awards, including the IEEE Roger W. Brockett Control Systems Award, Richard E. Bellman Control Heritage Award, Bode Lecture Prize, SIAM Reid Prize, ASME Oldenburger Medal, Nyquist Lecture Prize, Paynter Outstanding Investigator Award, and the Ragazzini Education Award and the IFAC Nonlinear Control Systems Award. 
		He serves as Editor-in-Chief of IEEE Transactions on Automatic Control.
	\end{IEEEbiography}
	
	\vspace{-0.7cm}
	\begin{IEEEbiography}[{\includegraphics[width=1.09in,height=1.25in,clip,keepaspectratio]{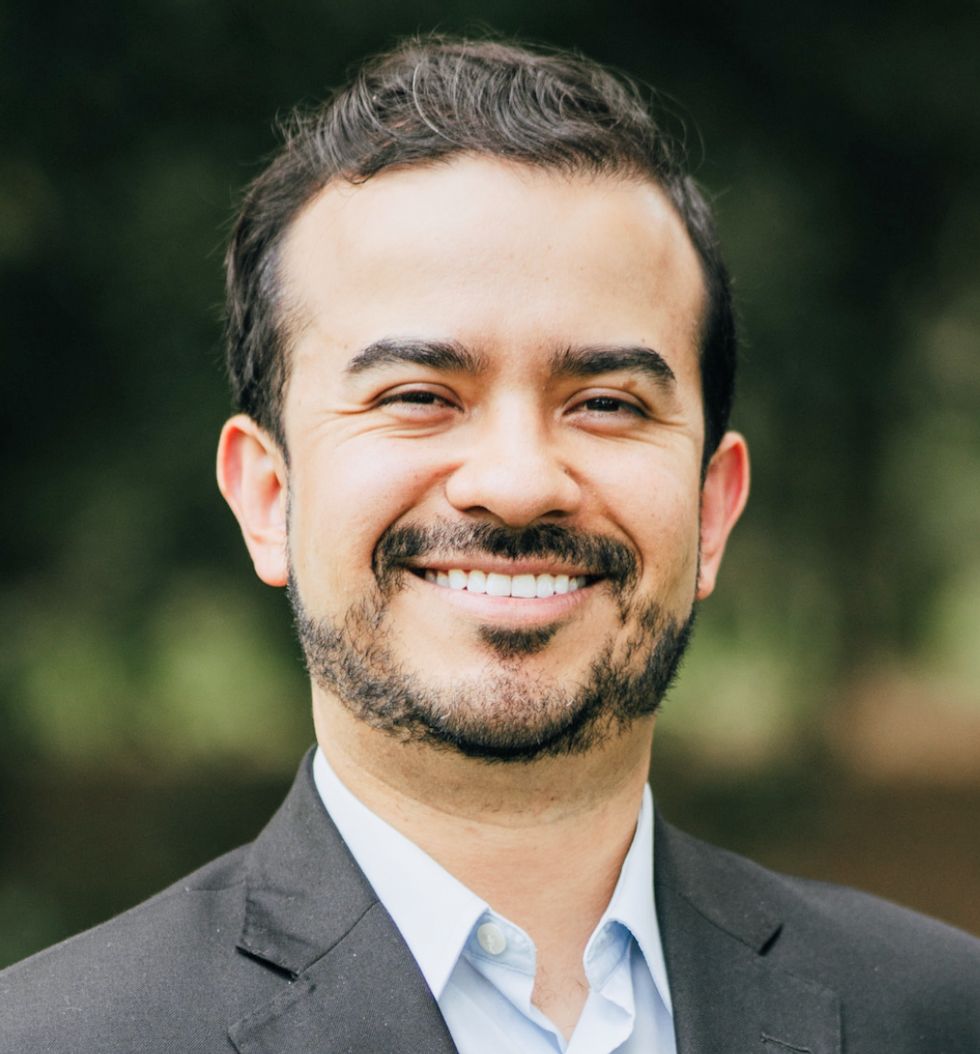}}]{Jorge I. Poveda} is Associate Professor in the Department of Electrical and Computer Engineering at the University of California, San Diego, La Jolla, CA, USA. He received double B.S. degrees in Mechanical Engineering and Electronics Engineering, both from the University of Los Andes, Colombia. Afterwards, he received his M.Sc. and Ph.D. degrees in Electrical and Computer Engineering from the University of California, Santa Barbara in 2016 and 2018, respectively, followed by a Postdoctoral position at Harvard University. He was also a Research Intern at the Mitsubishi Electric Research Laboratories in 2016 and 2017. He has received several awards, including the NSF CRII and CAREER Awards, the AFOSR and SHPE Young Investigator Awards, and the 2023 Donald P. Eckman Award from the American Automatic Control Council. 
	\end{IEEEbiography}
\end{document}